\documentclass[11pt,reqno]{amsart}
\usepackage[margin=3.6cm]{geometry}
\usepackage{graphicx}
\usepackage{amssymb}
\usepackage{bbm}
\usepackage[percent]{overpic}

\newcommand{\R}{{\mathbb R}}

\newcommand{\C}{{\mathbb C}} 
\newcommand{\Z}{{\mathbb Z}}

\newcommand{\ep}{\varepsilon}

\newtheorem{theorem}{Theorem}
\newtheorem{corollary}[theorem]{Corollary}
\newtheorem{proposition}[theorem]{Proposition}
\newtheorem{lemma}[theorem]{Lemma}

\theoremstyle{definition}
\newtheorem{definition}{Definition}
\newtheorem{remark}{Remark}

\begin{document}

\title[Zero set of monochromatic random waves]{Topology and nesting of the zero set components  of  monochromatic random waves}

\begin{abstract}
This paper is dedicated to the study of the topologies and nesting configurations of the components of the zero set of monochromatic random waves. We prove that the probability of observing any diffeomorphism type, and any nesting arrangement, among the zero set components is strictly positive for waves of large enough frequencies. Our results are a consequence of building Laplace eigenfunctions in Euclidean space whose zero sets have a component with prescribed topological type, or an arrangement of components with prescribed nesting configuration.
\end{abstract}

\author[Y. Canzani]{Yaiza Canzani}
\author[P. Sarnak]{Peter Sarnak}

\address[Y. Canzani]{ University of North Carolina at Chapel Hill.\medskip}
 \email{canzani@email.unc.edu}
\address[P. Sarnak]{ Institute for Advanced Study and  Princeton University.\medskip}
\email{sarnak@math.ias.edu}


\maketitle

\section{Introduction}

For $n \geq 1$ let $E_1(\R^n)$ denote the linear space of entire (real valued) eigenfunctions $f$ of the Laplacian $\Delta$ whose eigenvalue is $1$
\begin{equation}\label{E:0}
\Delta f + f=0.
\end{equation}
The zero set  of $f$ is the set 
\[V(f)=\{x \in \R^n:\; f(x)=0\}.\]
The zero set decomposes into a collection of connected components which we denote by $\mathcal C(f)$. Our interest is in the topology of $V(f)$ and of the members of $\mathcal C(f)$. Let $H(n-1)$ denote the (countable and discrete) set of diffeomorphism classes of compact connected smooth $(n-1)$-dimensional manifolds that can be embedded in $\R^n$. The compact components $c$ in $\mathcal C(f)$ give rise to elements $t(c)$ in $H(n-1)$ (here we are assuming that $f$ is generic with respect to a Gaussian measure so that $V(f)$ is smooth, see Section \ref{monochromatic}). The connected components of $\R^n \backslash V(f)$ are the nodal domains of $f$ and our interest is in their nesting properties, again for generic $f$. To each compact $c \in \mathcal C(f)$ we associate a finite connected rooted tree as follows. By the Jordan-Brouwer  separation Theorem \cite{Li} each component  $c \in \mathcal C(f)$ has an exterior and interior. We choose the interior to be the compact end. The nodal domains of $f$, which are in the interior of $c$, are taken to be the vertices of a graph. Two vertices share an edge if the respective nodal domains have a common boundary component (unique if there is one). This gives a finite connected rooted tree  denoted $e(c)$; the root being the domain adjacent to $c$ (see Figure 2). Let $\mathcal T$ be the collection (countable and discrete) of finite connected rooted trees. Our main results are that any topological type and any rooted tree can be realized by elements of $E_1(\R^n)$.


\begin{theorem}\label{main theorem R a}
 Given $t \in H(n-1)$ there exists $f \in E_1(\R^n)$ and $c\in \mathcal C(f)$ for which $t(c)=t$.
\end{theorem}	
\begin{theorem}\label{main theorem R b}
Given $T \in \mathcal T$ there exists $f \in E_1(\R^n)$ and $c \in \mathcal C(f)$ for which $e(c)=T$.
\end{theorem}	

Theorems \ref{main theorem R a} and \ref{main theorem R b}  are of basic interest in the understanding of  the possible shapes of nodal sets and domains of eigenfunctions in $\R^n$ (it applies equally well to any eigenfunction with eigenvalue $\lambda^2>0$ instead of $1$). Our main purpose however is to apply it to derive a basic property of the universal monochromatic measures $\mu_C$ and $\mu_X$ whose existence was proved  in  \cite{SW}.  We proceed to introduce these measures.

Let $(S^n,g)$ be the $n-$sphere endowed with a  smooth, Riemannian metric $g$. Our results apply equally well with $S^n$ replaced by any compact smooth manifold $M$; we restrict to $S^n$ as it allows for a very clean formulation. Consider an orthonormal basis $\{\varphi_j\}_{j=1}^\infty$ for $L^2(S^n,g)$ consisting of real-valued eigenfunctions, 
 $ \Delta_g \phi_j = -\lambda_j^2 \phi_j$.
A monochromatic random wave on $(S^n,g)$ is the Gaussian random  field $f=f_{\eta, \lambda}$
\begin{equation}
f:={D_{\eta, \lambda}^{-1/2}}\sum_{\lambda_j\in [\lambda,  \lambda+\eta_\lambda]} a_j \phi_j,\label{E:RWdef}
\end{equation}
where the $a_j$'s are real valued i.i.d standard Gaussians,  $a_j\sim N(0,1)_{\R}$,  $\eta_\lambda=\eta(\lambda)$ is a non-negative function satisfying $\eta(\lambda)=o(\lambda)$ as $\lambda \to \infty$, and $D_{\eta, \lambda}=\#\{j:\, \lambda_j \in [\lambda, \lambda+\eta]\}$.
 When choosing $\eta \equiv 0$ the $\lambda$'s we consider in forming the $f_{\eta, \lambda}$'s  are the square roots of the Laplace eigenvalues. 
To a monochromatic random wave we associate its (compact) nodal set $V(f)$ and a corresponding finite set of nodal domains. The connected components of $V(f)$ are denoted by $\mathcal C(f)$ and each $c \in \mathcal C(f)$ yields a $t(c) \in H(n-1)$. Each $c\in \mathcal C(f)$ also gives a tree end $e(c)$ in $T$ which is chosen to be the smaller of the two rooted trees determined by the inside and outside of $c \subset S^n$. The topology of $V(f)$ is described completely by the probability measure $\mu_{\mathcal C(f)}$ on $H(n-1)$ given by 
\[\mu_{C\mathcal (f)}:= \frac{1}{|\mathcal C(f)|} \sum_{c \in \mathcal C(f)} \delta_{t(c)},\]
where $\delta_t$ is a point mass at $t\in H(n-1)$.
Similarly, the distribution of nested ends of nodal domains of $f$ is described by the measure $\mu_{X(f)}$ on $T$ given by
\[\mu_{X(f)}:= \frac{1}{|\mathcal C(f)|} \sum_{c \in \mathcal C(f)} \delta_{e(c)},\]
with $\delta_{e}$ is the point mass at $e \in \mathcal T$.

 The main theorem in \cite{SW} asserts that there exist probability measures $\mu_{\mathcal C}$ and $\mu_X$ on $H(n-1)$ and $\mathcal T$ respectively to which $\mu_{\mathcal C(f)}$ and $\mu_{X(f)}$  approach as  $\lambda \to \infty$, for almost all $f=f_{\eta, \lambda}$, 
 provided one has that for every $x_0 \in S^n$
\begin{equation}\label{E:CovConv1}
\sup_{u, v \in B(0,{r_\lambda})}  \left | \partial_u^k\partial_v^j \,\left[ \text{Cov}(f_{\eta,\lambda}^{x_0}(u),f_{\eta,\lambda}^{x_0}(v))- \text{Cov}(f_\infty^{x_0}(u), f_\infty^{x_0}(v)) \right]\big. \right|=o(1),
\end{equation}
as $\lambda\to \infty$. 
Here, $r_\lambda=o(\lambda)$,  $f_{\eta, \lambda}^{x_0}:T_{x_0}S^n \to \R$ is the localized wave on $T_{x_0}S^n$ defined as  $f_{\eta, \lambda}^{x_0}(u)= f_{\eta, \lambda}\left(\exp_{x_0}(\tfrac{u}{\lambda}) \right)$, and  $f_\infty^{x_0}$ is the Gaussian random field on $T_{x_0}S^n$ characterized by the covariance kernel  $\text{Cov}(f_\infty^{x_0}(u), f_\infty^{x_0}(v) )=\int_{S_{x_0}S^n} e^{i\langle{u-v}, {w}\rangle_{g_{x_0}}}dw$ (see Section \ref{monochromatic}).

 The  probability measures  $\mu_{\mathcal C}$ and $\mu_X$ are universal in that they only depend on the dimension $n$ of $M$.
 
Monochromatic random waves on the $n$-sphere equipped with the round metric are known as random spherical harmonics whenever $\eta\equiv 0$. It is a consequence of the Mehler-Heine \cite{Meh} asymptotics that  they satisfy condition \eqref{E:CovConv1} for all $x_0 \in S^n$.  Also, on any $(S^n,g)$ the fields $f_{\eta, \lambda}$ with $\eta \to \infty$ satisfy condition \eqref{E:CovConv1} for all $x_0 \in S^n$. Finally, monochromatic random waves  $f_{\eta, \lambda}$  on $(S^n,g)$ with $\eta \equiv c$, for some  $c>0$, satisfy condition  \eqref{E:CovConv1} for every $x_0 \in S^n$ satisfying that the set of geodesic loops that close at $x_0$ has measure $0$ (see \cite{CH}).
 On general manifolds one can define monochromatic random waves just as in $(S^n,g)$.
 Monochromatic random waves with $\eta\equiv 0$ on the $n$-torus are known as arithmetic random waves. They satisfy condition \eqref{E:CovConv1} for all $x_0 \in \mathbb T^n$ if $n\geq 5$, and on $\mathbb T^n$ with $2\leq n \leq 4$ provided we work with a density one subsequence of $\lambda$'s \cite{EEH}.
 On general $(M,g)$ monochromatic random waves with $\eta \equiv c$, for some  $c>0$, satisfy condition  \eqref{E:CovConv1} for every $x_0 \in M$ satisfying that the set of geodesic loops that close at $x_0$ has measure $0$ (see \cite{CH}).
 Examples of such manifolds are surfaces without conjugate points, or manifolds whose sectional curvature is negative everywhere.

Our main application of Theorems \ref{main theorem R a} and  \ref{main theorem R b} is the following result.

\begin{theorem}\label{main theorem}Let $(S^n,g)$ be the $n$-sphere equipped with a smooth Riemannian metric. Let $\mu_{\mathcal C}$ and  $\mu_X$ be the limit  measures (introduced in \cite{SW}) arising from monochromatic random waves on $(S^n,g)$ for which  condition \eqref{E:CovConv1} is satisfied for every $x_0 \in S^n$. 
\begin{enumerate}
\item [(i)] 
The support of $\mu_{\mathcal C}$ is  $H(n-1)$. That is, every atom of $H(n-1)$ is positively charged by $\mu_{\mathcal C}$.\\
\item [(ii)] The support of $\mu_X$ is all of $\mathcal T$. That is, every atom of $\mathcal T$ is positively charged by $\mu_X$.
\end{enumerate}
\end{theorem}

\begin{remark} Theorem \ref{main theorem} asserts that every topological type that can occur will do so with a positive probability for the universal distribution of topological types of random monochromatic waves in \cite{SW}. The reduction from Theorems \ref{main theorem R a} and \ref{main theorem R b}  to  Theorem \ref{main theorem} is abstract and is based on the `soft' techniques in \cite{NS, SW} (see also Section \ref{monochromatic}). In particular, it offers us no lower bounds for these probabilities. Developing such lower bounds is an interesting problem. The same applies to the tree ends.
\end{remark}

\begin{remark}
Theorem \ref{main theorem} holds for monochromatic random waves on general compact, smooth, Riemannian manifolds $(M,g)$ without boundary.  Part (i) actually holds without modification. The reason why we state the result on the round sphere $S^n$ is that, by the Jordan-Brouwer separation Theorem \cite{Li}, on $S^n$ every component of the zero set separates $S^n$ into two distinct components. This gives that the nesting graph for the zero sets is a rooted tree. On general $(M,g)$ this is not necessarily true, so there is no global way to define a tree that describes the nesting configuration of the zero set in all of $M$, for all $c \in C(f)$. However, according to \cite{NS2} almost all $c$'s localize to small coordinate patches and hence our arguments apply.\end{remark}

We end the introduction with an outline of the paper. Theorem \ref{main theorem R a}  for $n=3$ (which is the first interesting case) is proved in \cite{SW} by deformation of the eigenfunction 
\begin{equation}\label{u}
u(x,y,z)=\sin (\pi x) \sin (\pi  y) + \sin (\pi x) \sin(\pi  z) + \sin (\pi y) \sin (\pi z).
\end{equation}
The proof exploits that the space $H(2)$ is simply the set of orientable compact surfaces which are determined by their genus. So in engineering a component of  a deformation of $f$ to have a given genus it is clear what to aim for in terms of how the  singularities (all are conic) of $f=0$ resolve. For $n\geq 4$, little is known about the space $H(n-1)$ and we proceed in Section \ref{topology} quite differently. We apply Whitney's approximation Theorem to realize $t$ as an embedded real analytic submanifold of $\R^n$. Then, following some techniques in \cite{EP} we find suitable approximations of $f\in E_1(\R^n)$ and whose zero set contains a diffeomorphic copy of $t$. The construction of $f$ hinges on the Lax-Malgrange Theorem and Thom's Isotopy Theorem. As far as Theorem \ref{main theorem R b}, the case $n=2$ is resolved in \cite{SW} using a deformation of $\sin(\pi x)\sin(\pi y)$ and a combinatorial chess board type argument. In higher dimensions, for example $n=3$ we proceed in Section \ref{nesting} by deforming
\begin{equation}\label{u2}
u(x,y,z)=\sin(\pi x)\sin (\pi y)\sin (\pi z).
\end{equation}
This $f$ has enough complexity (as compared to the $u$ in \eqref{u}) to produce all elements in $\mathcal T$ after deformation. However, it is much more difficult to study. Unlike \eqref{u} or  $\sin(\pi x)\sin(\pi y)$, the zero set $u^{-1}(0)$ in \eqref{u2} has point and $1$-dimensional edge singularities.  The analysis of its resolution under deformation requires a lot of care, especially as far as engineering elements of $\mathcal T$. The pay off as we noted is that it is rich enough to prove Theorem \ref{main theorem R b}. 

In Section \ref{monochromatic} we review some of the theory of monochromatic Gaussian fields and their representations. Section \ref{topology} is devoted to the proof of Theorem \ref{main theorem R a}. Section \ref{nesting} is devoted to the proof of Theorem \ref{main theorem R b}. The latter begins with an interpolation theorem of Mergelyan type, for elements in $E_1(\R^n)$. We use that to engineer deformations of \eqref{u2} which achieve the desired tree end, this being the most delicate aspect of the paper.

\section{Monochromatic Gaussian waves}\label{monochromatic}
Our interest is in the monochromatic Gaussian field on $\R^n$ which is a special case of the band limited Gaussian fields considered in \cite{SW}, and which is fundamental in the proof of \cite[Theoem 1.1]{SW}. For $0 \leq \alpha \leq 1$, define the annulus $A_\alpha=\{ \xi \in \R^n: \; \alpha \leq |\xi| \leq 1\}$ and let $\sigma_\alpha$ be the Haar measure on $A_\alpha$ normalized so that $\sigma_\alpha(A_\alpha)=1$. Using that the transformation $\xi \mapsto -\xi$ preserves $A_\alpha$  we choose a real valued orthonormal basis $\{\phi_j\}_{j=1}^\infty$ of $L^2( A_\alpha, \sigma_\alpha)$ satisfying 
\begin{equation}\label{E: -xi}
\phi_j(-\xi)= (-1)^{\eta_j} \phi_j(\xi), \qquad \quad \eta_j \in \{0,1\}.
\end{equation}
The band limited Gaussian field $H_{n,\alpha}$ is defined to be the random real valued functions $f$ on $\R^n$ given by 
\begin{equation}\label{E: f, b}
f(x)=\sum_{j=1}^\infty b_j\, i^{\eta_j}\, \widehat {\phi_j}(x)
\end{equation}
where
\begin{equation}\label{E: hat phi}
\widehat {\phi_j}(x)= \int_{\R^n} \phi_j(\xi) e^{- i \langle x, \xi\rangle} d \sigma_\alpha(\xi)
\end{equation}
and the $b_j$'s are identically distributed, independent, real valued, standard Gaussian variables. We note that the field $H_{n, \alpha}$ does not depend on the choice of the orthonormal basis $\{\phi_j\}$.

The distributional identity $\sum_{j=1}^\infty  \phi_j(\xi)  \phi_j(\eta)=\delta(\xi-\eta)$ on $A_\alpha$ together with \eqref{E: -xi} lead to the explicit expression for the covariance function:
 \begin{equation}\label{E: covariance}
 \text{Cov}(x,y):=\mathbb E_{H_{n,\alpha}}(f(x)f(y))=\int_{\R^n} e^{i \langle x-y , \xi \rangle} d\sigma_\alpha(\xi).
 \end{equation}
From \eqref{E: covariance}, or directly from \eqref{E: f, b}, it follows that almost all $f$'s in $H_{n,\alpha}$ are analytic in $x$ \cite{AT}. For the monochromatic case $\alpha=1$ we have 
\begin{equation}\label{E: Covariance}
 \text{Cov}(x,y)=\frac{1}{(2\pi)^{\frac{n}{2}}}\frac{J_{\nu}(|x-y|)}{ |x-y|^{\nu}}, 
 \end{equation}
 where to ease notation we have set 
  \[\nu:=\frac{n-2}{2}.\]
 
 In this case there is also a natural choice of a basis for $L^2(S^{n-1},d\sigma)=L^2(A_1, \mu_1)$ given by spherical harmonics.  Let $\{Y^\ell_m\}_{m=1}^{d_\ell}$ be a real valued basis for the space of spherical harmonics  $\mathcal E_\ell(S^{n-1})$ of eigenvalue $\ell(\ell+n-2)$, where $d_\ell=\dim \mathcal E_\ell(S^{n-1})$. 
We compute the Fourier transforms for the elements of this basis.
\begin{proposition} For every  $\ell \geq 0$ and $m=1, \dots, d_\ell$, we have 
\begin{equation}\label{E: f.t.  harmonic}
 \widehat{Y^\ell_m}(x)=(2\pi)^{\frac{n}{2}}\,i^\ell\, Y^\ell_m \left(\frac{x}{|x|}\right) \frac{J_{\ell + \nu}(|x|)}{|x|^{\nu}}.
\end{equation}
\end{proposition}

\begin{proof}
We give a proof using the theory of point pair invariants \cite{Sel} which places such calculations in a general and conceptual setting. The sphere $S^{n-1}$ with its round metric is a rank $1$ symmetric space and $\langle \dot x, \dot y \rangle$ for $\dot x, \dot y \in S^{n-1}$ is a point pair invariant (here $\langle \, , \, \rangle$ is the standard inner product on $\R^n$ restricted to $S^{n-1}$). Hence, by the theory of these pairs we know that for every function $h:\R \to \C$ we have
\begin{equation}\label{E: int=lambda}
\int_{S^{n-1}} h(\langle \dot x ,\dot y \rangle ) \,Y (\dot y) \,d\sigma(\dot y)= \lambda_h (\ell) Y (\dot x), 
\end{equation}
where $Y$ is any spherical harmonic of degree $\ell$ and $\lambda_h(\ell)$ is the spherical transform. The latter can be computed explicitly using the zonal spherical function of degree $\ell$.
Fix any $\dot x \in S^{n-1}$  and  let $Z^\ell_{\dot x}$ be the unique spherical harmonic of degree $\ell$ which is rotationally invariant by motions of $S^{n-1}$ fixing $\dot x$ and so that $Z^\ell_{\dot x}(\dot x)=1$. Then, 
\begin{equation}\label{E: lambda}
\lambda_h(\ell)=\int_{S^{n-1}} h(\langle \dot x, \dot y \rangle) Z^\ell_{\dot x}(\dot y)\, d\sigma(\dot y).
\end{equation}
The function $Z^\ell_{\dot x}(\dot y)$ may be expressed in terms of the Gegenbauer polynomials  \cite[(8.930)]{GR} as
\begin{equation}\label{E: zonal, gegenb}
 Z^\ell_{\dot x}(\dot y)=\frac{C_\ell^\nu \big(\big \langle \dot x, \dot y \big \rangle \big) }{C_\ell^{\nu}(1) }.
\end{equation}
Now, for $x \in \R^n$,
\[ \widehat{Y^\ell_m}(x)=\int_{S^{n-1}} h_x\big( \big \langle \tfrac{x}{|x|}, \dot y  \big \rangle \big)\, Y^\ell_m(\dot y) d\sigma(\dot y),\]
where we have set $h_x(t)=e^{-i|x| t}$. Hence, by \eqref{E: int=lambda} we have
\[ \widehat{Y^\ell_m}(x)= \lambda_{h_x} (\ell) \, Y^\ell_m\big( \tfrac{x}{|x|} \big),\]
with 
\begin{equation}\label{E: computation of lambda}
\lambda_{h_x}(\ell)
= \int_{S^{n-1}} e^{-i |x| \big \langle  \tfrac{x}{|x|}\, ,\, \dot y \big \rangle }Z^\ell_{\dot x}(\dot y)\, d\sigma(\dot y)
= \frac{\text{vol}(S^{n-2})}{C_\ell^\nu(1) } \int_{-1}^1 e^{-it|x|} \, C^\nu_\ell(t) (1-t^2)^{\nu -\frac{1}{2}} \, dt.
\end{equation}
The last term in \eqref{E: computation of lambda} can be computed using   \cite[(7.321)]{GR}. This gives
\[\lambda_{h_x}(\ell)=(2\pi)^{\frac{n}{2}}\, i^\ell\, \frac{J_{\ell+\nu}(|x|)}{|x|^\nu},\]
as desired.

\end{proof} 

\begin{corollary}
The monochromatic Gaussian ensemble $H_{n,1}$ is given by random $f$'s  of the form
\begin{equation*}
{f}(x)=(2\pi)^{\frac{n}{2}} \sum_{\ell=0}^\infty \sum_{m=1}^{d_{\ell}} b_{\ell,m} \,  Y^\ell_m \left(\frac{x}{|x|}\right) \frac{J_{\ell + \nu}(|x|)}{|x|^{\nu}},
\end{equation*}
where the $b_{\ell, m}$'s are i.i.d  standard  Gaussian variables.
\end{corollary}

The functions $x \mapsto Y^\ell_m \left(\frac{x}{|x|}\right) \frac{J_{\ell + \nu}(|x|)}{|x|^{\nu}}$, $x \mapsto e^{i \langle x, \xi\rangle}$ with $|\xi|=1$, and those in \eqref{E: f, b} for which the series converges rapidly (eg. for almost all $f$ in $H_{n,1}$), all satisfy \eqref{E:0}, that is $f \in E_1(\R^n)$.
In addition, consider the subspaces $P_1$ and $T_1$ of $E_1(\R^n)$ defined by
\[P_1:= \text{span}\left\{x \mapsto Y^\ell_m \left(\frac{x}{|x|}\right) \frac{J_{\ell + \nu}(|x|)}{|x|^{\nu}}:\; \, \ell \geq 0, \; m=1, \dots, d_\ell \right\},\] 
\[T_1:= \text{span}\left\{x\mapsto\frac{e^{i \langle x, \xi\rangle} +e^{-i \langle x, \xi\rangle} }{2}\, ,\,x\mapsto \frac{e^{i \langle x, \xi\rangle} -e^{-i \langle x, \xi\rangle} }{2i}:\;\; |\xi|=1 \right\} .\] 

\begin{proposition}\label{E: approximation}
Let  $f \in E_1(\R^n)$ and let $K \subset \R^n$ be a compact set. Then, for any $t \geq 0$ and $\ep>0$ there are $g \in P_1$ and $h \in T_1$ such that 
\[\|f-g\|_{C^t(K)}<\ep \qquad \text{and}\qquad \|f-h\|_{C^t(K)}<\ep.\]
That is, we can approximate $f$ on compact subsets in the $C^t$-topology  by elements of $P_1$ and $T_1$ respectively. 
\end{proposition}

\begin{proof}
Let $f\in E_1$. Since $f$ is analytic we can expand it in a rapidly convergent series in the $Y^\ell_m$'s. That is, 
\[f(x)=\sum_{\ell=0}^\infty \sum_{m=1}^{d_\ell} a_{m,\ell}(|x|)Y^\ell_m( \tfrac{x}{|x|}).\]
Moreover, for $r>0$,
\begin{equation}\label{E: fourier a's}
\int_{S^{n-1}}|f(r \dot x)|^2\, d\sigma(\dot x)= \sum_{\ell=0}^\infty \sum_{m=1}^{d_\ell} |a_{m,\ell}(r)|^2.
\end{equation}
In polar coordinates, $(r, \theta) \in (0, +\infty)\times S^{n-1}$,   the Laplace operator in $\R^n$  is given by  
\[\Delta= \partial_r^2 + \frac{n-1}{r} \partial_r  +\frac{1}{r^2} \Delta_{S^{n-1}},\] 
and hence for each $\ell,m$ we have that 
 \begin{equation}\label{E: a}
 r^2a_{m,\ell}''(r) + (n-1) r a_{m,\ell}'(r) +(r^2- \ell(\ell+n-2))a_{m,\ell}(r)=0.
 \end{equation}
where $\ell$ is some positive integer.
There are two linearly independent solutions to \eqref{E: a}. One is $r^{-\nu}J_{\ell+\nu}(r)$ and the other blows up as $r \to 0$. Since the left hand side of \eqref{E: fourier a's} is finite as $r\to 0$, it follows that the $ a_{m,\ell}$'s cannot pick up any component of the blowing up solution. That is, for $r \geq 0$
\[ a_{m,\ell}(r)= c_{\ell,m}  \frac{J_{\ell +\nu}(r)}{r^\nu},\]
for some $c_{m,\ell} \in \R$.
Hence, 
\begin{equation}\label{E:  expansion for f}
{f}(x)=\sum_{\ell=0}^\infty \sum_{m=1}^{d_{\ell}} c_{\ell,m} \; Y^\ell_m \left(\frac{x}{|x|}\right) \frac{J_{\ell + \nu}(|x|)}{|x|^{\nu}}.
\end{equation}
Furthermore, this series converges absolutely and uniformly on compact subsets, as also do its derivatives. Thus, $f$ can be approximated by members of $P_1$ as claimed, by simply truncating the series in \eqref{E:  expansion for f}. 

To deduce the same for $T_1$ it suffices to approximate each fixed $Y^\ell_m \left(\frac{x}{|x|}\right) \frac{J_{\ell + \nu}(|x|)}{|x|^{\nu}}$. To this end let $\xi_1, -\xi_1, \xi_2, -\xi_2,\dots ,\xi_N, -\xi_N$ be a sequence of points in $S^{n-1}$ which become equidistributed with respect to $d\sigma$  as $N \to \infty$.
Then, as $N\to \infty$,
\begin{equation}\label{E: limit}
 \frac{1}{2N} \sum_{j=1}^N \left( e^{-i \langle x, \xi_j \rangle} Y^\ell_m(\xi_j)+ (-1)^\ell e^{i \langle x, \xi_j \rangle} Y^\ell_m(\xi_j)\right) \longrightarrow \int_{S^{n-1}} e^{-i \langle x, \xi \rangle} Y^\ell_m(\xi) \,d\sigma(\xi).
 \end{equation}
The proof follows since $(2\pi)^{\frac{n}{2}}\, i^\ell\, Y^\ell_m \left(\frac{x}{|x|}\right) \frac{J_{\ell + \nu}(|x|)}{|x|^{\nu}}= \int_{S^{n-1}} e^{-i \langle x, \xi \rangle} Y^\ell_m(\xi) \,d\sigma(\xi)$. Indeed, the convergence in \eqref{E: limit} is uniform over compact subsets in $x$. 
\end{proof}

\begin{remark}
For $\Omega \subset \R^n$ open, let $E_1(\Omega)$ denote the eigenfunctions on $\Omega$ satisfying $\Delta f(x)+f(x)=0$ for $x \in \Omega$. Any function $g$ on $\Omega$ which is a limit (uniform over compact subsets of $\Omega$) of members of $E_1$ must be in $E_1(\Omega)$. While the converse is not true in general, note that if $\Omega=B$ is a ball in $\R^n$, then the proof of Proposition \ref{E: approximation} shows that the uniform limits of members of $E_1$ (or $P_1$, or $T_1$) on compact subsets in $B$ is precisely $E_1(B)$. 
\end{remark}

With these equivalent means of approximating functions by suitable members of $H_{n,1}$, and particularly $E_1(\R^n)$, we are ready to prove Theorems \ref{main theorem R a} and \ref{main theorem R b}. Indeed, as shown in \cite{SW} the extension of condition $(\rho 4)$ of \cite[Theorem 1]{NS2} suffices. Namely, for $c \in H(n-1)$ it is enough to find an $f \in T_1$ with $f^{-1}(0)$ containing $c$ as one of its components for Theorem \ref{main theorem R a}, and for $T \in \mathcal T$ it suffices to find an $f \in T_1$ such that $e(c)=T$ for some component $c$ of $f^{-1}(0)$.

\section{Topology of the zero set components}\label{topology}

In this section we prove Theorem \ref{main theorem R a}.
By the discussion  above it follows that given  a representative $c$ of a class $t(c) \in H(n-1)$,  it suffices to find $f \in E_1(\R^n)$  for which $\mathcal C(f)$ contains a diffeomorphic copy of $c$. 

To begin the proof we claim that we may assume that $c$ is real analytic. Indeed, if we start with $\tilde{c}$ smooth, of the desired topological type, we may construct a tubular neighbourhood $V_{\tilde c}$ of $\tilde {c}$  and a smooth  function 
$$H_{\tilde c}:V_{\tilde c} \to \R \qquad \text{with}\qquad \tilde c= H_{\tilde c}^{-1}(0).$$ Note that without loss of generality we may assume that   $\inf_{x \in V_{\tilde c}}\|\nabla H_{\tilde c}(x)\|>0$.
Fix any $\epsilon>0$. We apply Thom's isotopy Theorem \cite[Thm 20.2]{AR} to obtain the existence of a constant $\delta_{\tilde c}>0$ 
so that for any function $F$ with $\|F-H_{\tilde c}\|_{C^1(V_{\tilde c})} <\delta_{\tilde c}$ there exists $\Psi_{F} :\R^n \to \R^n$ diffeomorphism with
$$\Psi_{F}(\tilde c)= F^{-1}(0) \cap V_{\tilde c}.$$
To construct a suitable $F$ we use Whitney's approximation Theorem \cite[Lemma 6]{Wh} which yields the existence of a real analytic approximation $F:V_{\tilde c} \to \R^{m_{\tilde c}}$  of $H_{\tilde c}$ that satisfies $\|F-H_{\tilde c}\|_{C^1(V_{\tilde c})}< \delta_{\tilde c}$. It follows that $\tilde c$ is diffeomorphic to $c:=\Psi_{F}(\tilde c)$ and $c$ is real analytic as desired. 

By the Jordan-Brouwer Separation Theorem \cite{Li}, the hypersurface $c$ separates $\R^n$ into two connected components. We write $A_c$ for the corresponding bounded component of $\R^n \backslash c$.
Let $\lambda^2$ be the first Dirichlet eigenvalue for the domain $A_c$ and let $h_\lambda$ be the corresponding eigenfunction: 
$$\begin{cases}
(\Delta + \lambda^2 )h_\lambda(x)=0& x \in \overline{A_c},\\
 h_\lambda(x)=0& x \in c.
\end{cases}$$
Consider the rescaled function $$h(x):=h_\lambda(x/\lambda),$$ defined on the rescaled domain $\lambda A_c:=\{x \in \R^n:  x/\lambda \in A_c\}$. 
Since $(\Delta + 1)h=0$ in $\overline{\lambda A_c}$, and $\partial (\lambda A_c)$ is real analytic, $h$ may be extended to some open set $B_c\subset \R^n$ with $\overline{\lambda A_c} \subset B_c$ so that
$$\begin{cases}
(\Delta + 1)h(x)=0& x \in B_c,\\
 h(x)=0& x \in  \lambda c,
\end{cases}$$
where $\lambda c$ is the rescaled hypersurface $\lambda c:= \{x \in \R^n:\; x/\lambda \in c\}$.
Note that since $h_\lambda$ is the first Dirichlet eigenfunction, then we know that there exists a tubular neighbourhood $V_{\lambda c}$ of $\lambda c$ on which  $\inf_{x \in V_{\lambda c}}\|\nabla h(x)\|>0$ (see Lemma 3.1 in \cite{BHM}). Without loss of generality assume that  $V_{\lambda c}\subset B_c$.

We apply Thom's isotopy Theorem \cite[Thm 20.2]{AR} to obtain the existence of a constant $\delta>0$ 
so that for any function $f$ with $\|f-h\|_{C^1(V_{\lambda c})} <\delta$ there exists $\Psi_f :\R^n \to \R^n$ diffeomorphism so that
$$\Psi_f(\lambda c)= f^{-1}(0) \cap V_{\lambda c}.$$
Since $\R^n \backslash B_c$ has no compact components,  Lax-Malgrange's Theorem \cite[p. 549]{Kr} yields the existence of a global solution $f:\R^n \to \R$ to the elliptic equation $(\Delta + 1)f=0$ in $\R^n$ with $$\|f-h\|_{C^1(B_c)}<\delta.$$
We have then constructed a solution to $(\Delta + 1)f=0$ in $\R^n$, i.e. $f\in E_1$,  for which $f^{-1}(0)$ contains a diffeomorphic copy of $c$ (namely, $\Psi_f(\lambda c))$. This concludes the proof of the theorem.

\qed

 We note that the problem of finding a solution  to $(\Delta+1)f=0$ for which $C(f)$ contains a diffeomorphic copy of $c$ is related to the work \cite{EP} of A.~ Enciso and D.~Peralta-Salas. In \cite{EP} the authors seek to find solutions to the problem $(\Delta -q)f=0$ in $\R^n$ so that $C(f)$  contains a diffeomorphic copy of $c$, where $q$ is a {nonnegative}, real analytic, potential and $c$ is a (possibly infinite) collection of compact or unbounded ``tentacled" hypersurfaces. The construction of the solution $f$ that we presented is shares ideas with \cite{EP}.  Since our setting and goals are simpler than theirs, the construction  of $f$  is much shorter and straightforward.


\section{Nesting of nodal domains} \label{nesting}

The proof of Theorem \ref{main theorem R b} consists in perturbing the zero set of the eigenfunction 
$u_0(x_1,\dots,x_n)=\sin(\pi x_1)\dots\sin(\pi x_n)$ so that the zero set of the perturbed function  will have the desired nesting. The nodal domains of $u_0$ build a n-dimensional chess board made out of unit cubes. By adding a small perturbation to $u_0$ the changes of topology in $u_0^{-1}(0)$ can only occur along the singularities of $u_0^{-1}(0)$. Therefore, we will build an eigenfunction $f$, satisfying $-\Delta f=f$, by prescribing it  along the singularities $L=\cup_{a,b \in \Z} \cup_{i,j=1,\; i\neq j}^n \{(x_1, \dots, x_n) \in \R^n:\; x_i=a, \;x_j=b \}$ of the zero set of $u_0$. We then construct a new eigenfunction $u_\ep = u_0 + \ep f$ which will have the desired nesting among a subset of its nodal domains.   The idea is to prescribe $f$ on the singularities of the zero set of $u_0$ in such a way that two adjacent cubes of the same sign will either glue or disconnect along the singularity.    The following theorem shows that one can always find  a solution $f$ to $-\Delta f=f$ with prescribed values on a set of measure zero (such as  $L$).
We prove this result following the first step of Carleson's proof \cite{Car} of Mergelyan's classical Theorem about analytic functions.

\begin{theorem}\label{T:perturbation}
Let $K \subset \R^n$  be a compact set with  Lebesgue measure $0$ and  so that $\R^n \backslash K$ is connected.   Then, for every $\delta>0$ and $h\in C_c^2(\R^n)$  there exists $f:\R^n \to \R$ satisfying 
\[-\Delta f=f  \quad \text{and} \quad   \sup_{ K}\{|f-h|+ \|\nabla f-\nabla h\|\} \leq \delta.\] 
\end{theorem}
\begin{remark}
In the statement of the theorem the function $h\in C_c^2(\R^n)$ can be replaced by $h\in C_c^1(\Omega)$, where $\Omega \subset \R^n$ is any open set with $K \subset \Omega$. This is because $C_c^2(\R^n)$ is dense in  $C_c^1(\Omega)$ in the $C^1$-topology.
\end{remark}
\begin{proof}
Consider the sets
\begin{align*}
\mathcal A&= \{(\phi, \partial_{x_1} \phi, \dots,  \partial_{x_n} \phi ):\; \phi \in \ker(\Delta+1)\},\\
\mathcal B&= \{(\phi, \partial_{x_1} \phi, \dots,  \partial_{x_n} \phi ):\; \phi \in C_c^2(\R^n)\},
\end{align*}
and write $\mathcal A_K, \mathcal B_K$ for the restrictions of $\mathcal A,\mathcal B$ to $K$.
Both  $\mathcal A_K$ and $\mathcal B_K$ are subsets of the Banach space $\oplus_{k=0}^n C(K)$, and clearly  $ \mathcal A_K\subset \overline{\mathcal B_K}^{\|\,\|_{C^0}}$. It follows that the claim in the theorem is equivalent to proving that 
\begin{equation}\label{E:goal2}
\mathcal B_K \subset  \overline{\mathcal A_K}^{\|\,\|_{C^0}} .
\end{equation}
To prove \eqref{E:goal2}, note that a distribution $D$ in the dual space $(\oplus_{k=0}^n C(K))^*$ can be identified with an $(n+1)$-tuple of measures $(\nu_0, \nu _1, \dots, \nu_n)$ with $\nu_j \in (C(K))^*$ for each $j=0, \dots n$. That is, for each $(\psi_0, \psi_1, \dots, \psi_n)\in \oplus_{j=0}^n C(K)$, 
\begin{equation}\label{E:D}
D(\psi_0, \psi_1, \dots, \psi_n)=  \sum_{j=0}^n \int_K  \psi_j \, d\nu_j.
\end{equation}
Since $\overline{\mathcal A_K}^{\|\,\|_{C^0}}=(\mathcal {A_K}^\perp )^\perp$, proving \eqref{E:goal2} is equivalent to showing that for each $D \in (\oplus_{k=0}^n C(K))^* $ satisfying
$D (\Phi ) =0 $ for all $\Phi \in \mathcal A_K,$
one has that
$D (\Phi ) =0 $ for all $\Phi \in \mathcal B_K.$ Using that each $D\in (\oplus_{k=0}^n C(K))^*$ is supported in $K$, we have reduced our problem to showing that 
\begin{align}\label{E:goal3}
\text{If}\;  D\in & (\oplus_{k=0}^n C(K))^*\; \; \text{satisfies} \;\;D (\Psi ) =0 \; \;\;\forall \Psi \in \mathcal A,\notag\\ &\qquad \qquad \text{then}\;\;D (\Phi ) =0 \; \;\;\forall \Phi \in \mathcal B.
\end{align}

We proceed to prove the claim in \eqref{E:goal3}. Fix $D\in (\oplus_{k=0}^n C(K))^*$ satisfying the assumption in  \eqref{E:goal3}.  Given  $\phi \in C^2_c(\R^n)$ we need to prove that $D(\phi, \partial_{y_1} \phi, \dots,  \partial_{y_n} \phi ) =0$.  Consider the fundamental solution \[N(x,y):= \frac{1}{n(n-2)\omega_n} \,\frac{1}{|x-y|^{n-2}},\]
where $\omega_n$ is the volume of the unit ball in $\R^n$. 
Note that there exists $C > 0$ so that
$\left|\frac{\partial N}{\partial y_j}(x,y) \right|< \frac{C}{|x-y|^{n-1}}$ for all  $j =0,\dots n.$
Therefore, for $y$ fixed, $N(x,y)$ and $\frac{\partial N}{\partial y_j}(x,y)$ are locally integrable in $\R^n$.
In particular, $N(x,y)\,|d\nu_0(y)|\, dx$ and $\frac{\partial N}{\partial y_j}(x,y)\, |d\nu_j(y)| \,dx$ are integrable on the product $K\times\R^n$, where the $\nu_j$'s are as in \eqref{E:D}.
Also, note that 
\[
\phi(y) = \int_{\R^n} (\Delta+1)\phi(x) N(x,y) dx
\quad
\text{and}
\quad
\frac{\partial \phi}{\partial y_j}(y) = \int_{\R^n} (\Delta+1)\phi (x) \frac{\partial N}{\partial y_j}(x,y) dx.
\]
By these observations, and since $K$ has measure zero, we may apply Fubini to get
\begin{align*}
&D(\phi, \partial_{y_1} \phi, \dots,  \partial_{y_n} \phi ) =\\
&=\int_K \phi(y)\, d\nu_0(y) + \sum_{j=1}^n \int_K \frac{\partial \phi}{\partial y_j}(y)\, d\nu_j(y)\\
&\!=\int_K  \int_{\R^n \backslash K} (\Delta+1)\phi(x) N(x,y) dx d\nu_0(y) + \sum_{j=1}^n \int_K\int_{\R^n \backslash  K} (\Delta+1)\phi (x) \frac{\partial N}{\partial y_j}(x,y) dxd\nu_j(y)\\
&\!= \int_{\R^n \backslash K} \int_K (\Delta+1)\phi(x) N(x,y) dx d\nu_0(y) + \sum_{j=1}^n  \int_{\R^n \backslash K} \int_K (\Delta+1)\phi (x) \frac{\partial N}{\partial y_j}(x,y) dx d\nu_j(y)\\
&= \int_{\R^n\backslash K} (\Delta+1)\phi(x) F(x) dx,
\end{align*}
where 
\[F(x):=\int_K N(x,y)\, d\nu_0(y) + \sum_{j=1}^n \int_K \frac{\partial N}{\partial y_j}(x,y)\, d\nu_j(y).\]  

The claim that $D(\phi, \partial_{y_1} \phi, \dots,  \partial_{y_n} \phi ) =0$  follows from the fact that   $F(x)=0$ for $x\in\R^3\setminus K$.  To see this,
 let  $R>0$ be large enough so that $K\subset B(0,R)$.  Then, for $x\in\R^n\backslash B(0,R)$, the map $\psi^x(y):=N(x,y)$ is in $\ker(\Delta+1)\vert_{B(0,R)}$. Applying  Proposition \ref{E: approximation}  we know that  there exists a sequence $\{\psi_\ell^x\}_{\ell} \subset \ker(\Delta+1)$ for which
\[\|\psi_\ell^x - \psi^x\|_{C^1(B(0, R))} \overset{\ell\to \infty}{\longrightarrow} 0.\] 
Hence, by the assumption in  \eqref{E:goal3},  for each $x\in\R^n\backslash B(0,R)$
\begin{align}
 0=D(\psi^x, \partial_{y_1} \psi^x, \dots,  \partial_{y_n} \psi^x)
 =\int_K N(x,y)\, d\nu_0(y) + \sum_{j=1}^n \int_K \frac{\partial N}{\partial y_j}(x,y)\, d\nu_j(y)=F(x). \label{E:F}
\end{align}
Now, the integral defining $F(x)$ converges absolutely for $x\in \R^n\setminus K$ and defines an analytic function of $x$ in this set. Since $F(x)$ vanishes for $x\in\R^n\backslash B(0,R)$, and $\R^n\setminus K$ is connected, it follows that 
\[F(x)=0 \quad \text{for all}\quad  x\in\R^n\setminus K,\]
as claimed.

\end{proof}

\subsection{Construction of the rough domains}

We will give a detailed proof Theorem \ref{main theorem R b} in $\R^3$ since in this setting it is easier to visualize how the argument works. In Section \ref{higher dimensions} we explain the modifications one needs to carry in order for the same argument to hold in $\R^n$. \\

\noindent Let $u_0:\R^3 \to \R$ be defined as \[u_0(x,y,z)=\sin (\pi x) \sin(\pi y)\sin (\pi z).\] Its nodal domains consist of a collection of cubes whose vertices lie on the grid $\Z^3$.
Throughout this note the cubes are considered to be closed sets, so faces and vertices are included.
We say that a cube is positive (resp. negative) if $u_0$ is positive (resp. negative) when restricted to it. 
We define the collection $\mathcal B^+$ of all sets $\Omega$ that are built as a finite union of cubes with the following two properties:
\begin{itemize}
\item  $\R^3 \backslash \Omega$ is connected. 
\item All the cubes in $\Omega$ that have a face in $\partial \Omega$ are positive. 
\end{itemize}
We define  $\mathcal B^-$ in the same way only that the faces in $\partial \Omega$ should belong to negative cubes.

\ \\{\bf Engulf operation.} Let $C \in \mathcal B^+$. We proceed to define the ``engulf" operation as follows. We define $E(C)$ to be the set obtained by adding to $C$ all the negative cubes that touch $C$, even if they share only one point with $C$. By construction   $E(C) \in \mathcal B^-$.  
If $C \in \mathcal B^-$, the set $E(C)$ is defined in the same form only that one adds positive cubes to $C$. In this case $E(C) \in \mathcal B^+$.
 
    \begin{center}
\begin{overpic}[height=4.5cm]{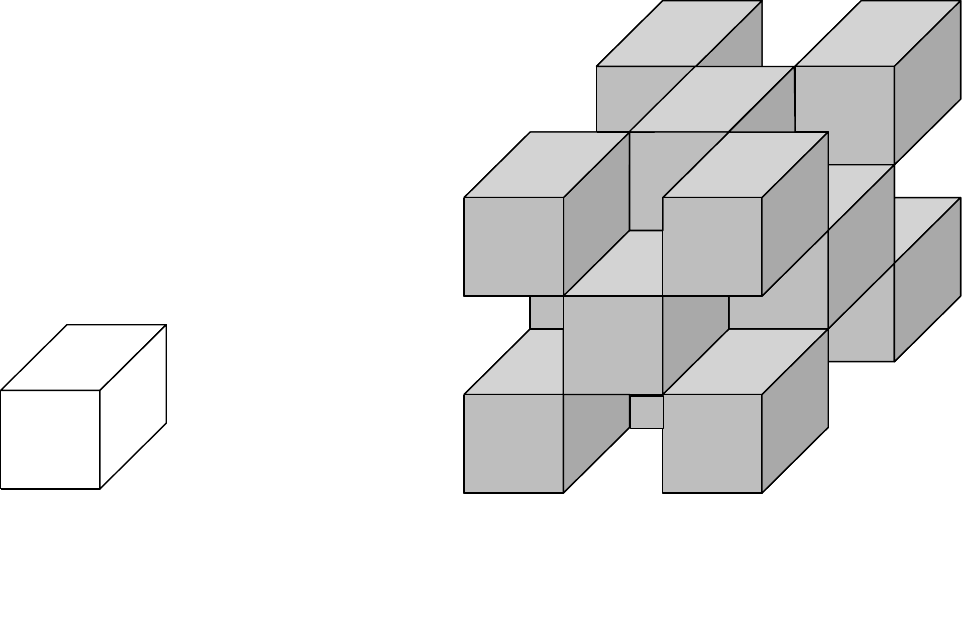}
\put(2,1){$C$ \hspace{3.5cm} $E(C)$}
\end{overpic}
\end{center}

\ \\{\bf Join operation.} Given $C \in \mathcal B^+ \cup \mathcal B^-$ we distinguish two vertices using the lexicographic order.  Namely, for any set of vertices $\Gamma \subset \Z^3$, for $i \in \{1,2,3\}$  we  set  
\[A_i^{\min}=\left\{(x_1^*, x_2^*,x_3^*) \in \Gamma: \;\; x_i^*= \min\{x_i: \; (x_1,x_2,x_3) \in \Gamma\}  \big. \right\}\subset \Z^3.\]
In the same way we define $A_i^{\max}$ replacing the minimum function above by the maximum one.
For  $C \in \mathcal B^+ \cup \mathcal B^-$, let $\Gamma_C= C \cap \Z^3$ be the set of vertices of cubes in $C$. We then set
\[ v_+(C)=A_1^{\max}(A_2^{\max}(A_3^{\max}(\Gamma_ C))) \qquad \text{and}\qquad v_-(C)=A_1^{\min}(A_2^{\min}(A_3^{\min}(\Gamma_ C))).\]
Given the vertex $v_+(C)$ we define the edge $e_+(C)$ to be the edge in $\partial C$ that has vertex $v_+(C)$ and is parallel to the $x$-axis. The edge $e_-(C)$ is defined in the same way.

We may now define the ``join" operation. Given $C_1 \in \mathcal B^+$ and $C_2 \in \mathcal B^+$ we define $J(C_1, C_2) \in \mathcal B^+$  as follows. Let $\tilde C_2$ be the translated copy of $C_2$ for which $e_+(C_1)$ coincides with $e_-(\tilde C_2)$.
We ``join" $C_1$ and $C_2$  as
\[J(C_1, C_2) = C_1 \cup \tilde C_2.\]   

In addition, for a single set $C$ we define $J(C)=C$, and if there are multiple sets $C_1, \dots, C_n$ we define \[J(C_1, \dots, C_n)=J(C_1,  J(C_2, J(C_3, \dots J(C_{n-1},C_n)))).\]

\ \\ \emph {Definition of the rough nested domains.}
Let $T_\infty:= \cup_{k=0}^\infty \mathbb N^k$. A rooted tree is characterized as a finite set of nodes $T \subset T_\infty$ satisfying that
\begin{align*}
&\bullet\; \emptyset \in T, \\
&\bullet\; (k_1, \dots, k_{\ell+1}) \in T \;\Longrightarrow\;  (k_1, \dots, k_{\ell}) \in T, \\
&\bullet\; (k_1, \dots, k_{\ell}, j) \in T \;\Longrightarrow\; (k_1, \dots, k_{\ell}, i) \in T \quad\text{for all $i \leq j$.}
\end{align*}
To shorten notation, if $v \in T$ is a node with $N$ children, we denote the children by $(v,1), \dots, (v,N)$. 

\begin{figure}[h!]
\includegraphics[height=6cm]{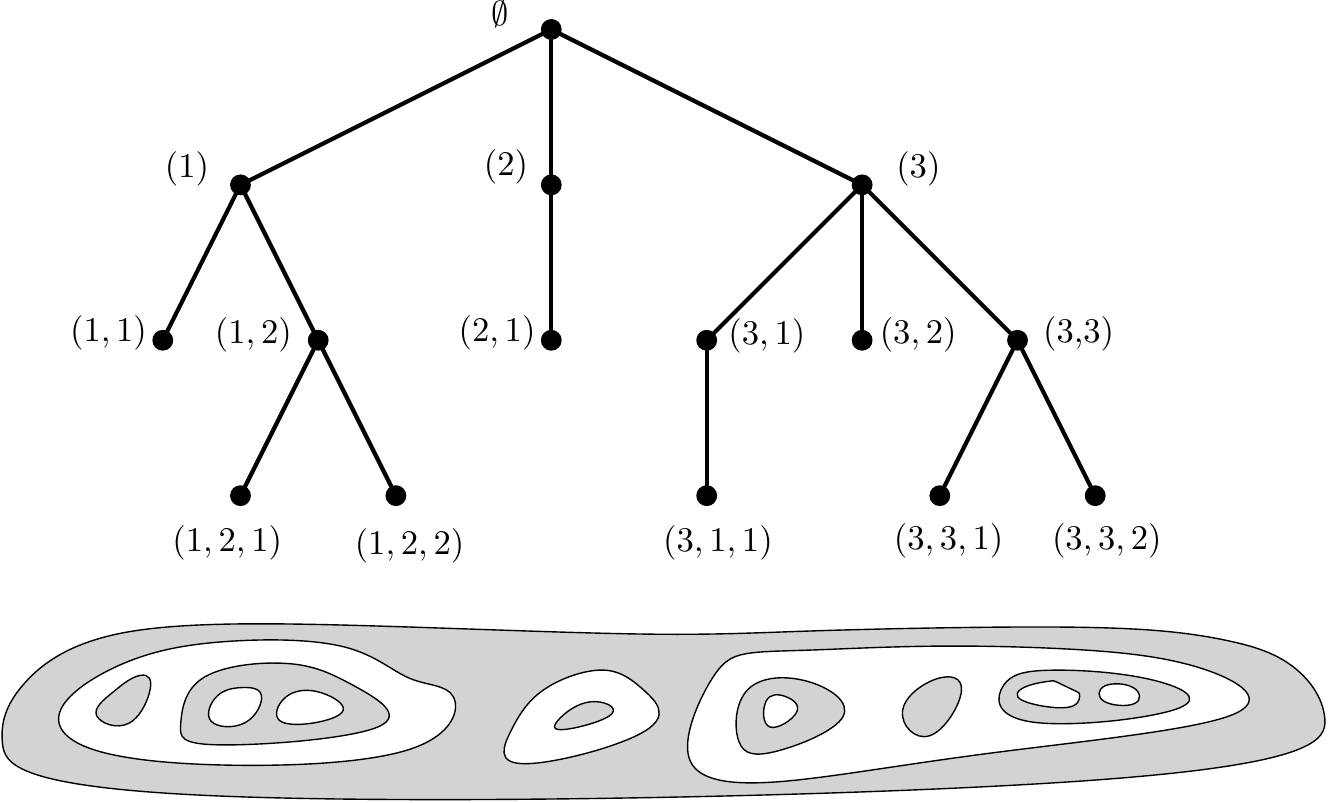}
\caption{{\small Example of a tree and  a transversal cut of the corresponding nesting of nodal domains. All the domains in figures below are labeled after this example.}}
 \end{figure}
 Given a tree $T$ we associate to each node  $v\in T$ a structure $C_{v} \subset \R^3$ defined as follows.  If the node   $v \in T$ is a leaf, then $C_{v}$ is a cube of the adequate sign. For the rest of the nodes we set
\[ C_{ v}= J\left(E( C_{ (v,1)}), \dots, E( C_{ (v,N)}) \big.\right),\]
where $N$ is the number of children of the node $v$.
It is convenient to identify the original structures $E( C_{ (v,j)})$ with the translated ones $\tilde E( C_{ (v,j)})$ that are used to build  $C_{v}$. After this identification, 
\[ C_{v}:= \bigcup_{j=1}^N  E( C_{ (v,j)}).\]
\begin{figure}[h!]
\begin{overpic}[height=5.5cm]{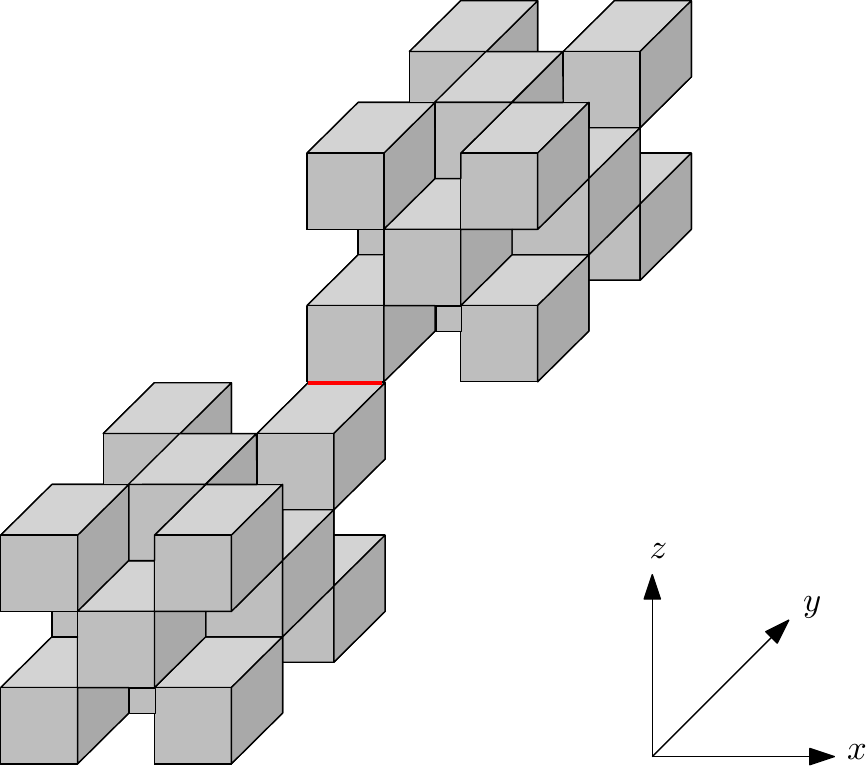}
\put(15,80){$\tilde{E}(C_1)$}
\put(-29,10){$E(C_2)$}
\end{overpic}
\caption{{\small This picture shows $J(E(C_{1}), E(C_{2}))$. The edge $e_+(E(C_{2})=e_-(\tilde E(C_{1})$ is depicted in red.   }}
\label{join}
 \end{figure}
%

\subsection{Building the perturbation}\label{S: perturbation}
Let  $v \in T$ be a node with $N$ children. We define the set of edges connected to $C_v$ on which the perturbation will be defined.\smallskip
\begin{itemize}
\item  We let  $\mathcal E_{\text{join}}(C_v)$ be the set of edges  in  $\partial C_v$ through which the structures $\{E( C_{(v,j)})\}_{j=1}^N$ are joined. We will take these edges to be open. That is, the edges in $\mathcal E_{\text{join}}(C_v)$ do not include their vertices.\\
\item We  let
$\mathcal E_{\text{ext}} (C_v)$ be the set of edges in $ \mathcal S_{\text{ext}}(C_v) $ that are not in  $\mathcal E_{\text{join}}(C_v)$. Here $\mathcal S_{\text{ext}} (C_v)$ is the surface
\begin{equation}\label{surface}
 \mathcal S_{\text{ext}} (C_v):=\{ x \in \R^3:\; d_{\text{max}} \big(x,  \cup_{j=1}^NC_{v,j} \big)=1\}.
 \end{equation}  If $v$ is a leaf, we set $\mathcal S_{\text{ext}} (C_v)=\partial C_v$.  All the edges in $\mathcal E_{\text{ext}} (C_v)$ are taken to be closed (so they include the vertices). \\
\item We let  $\mathcal E_{\text{int}}(C_v)$ be the set of edges that connect $ \mathcal S_{\text{ext}} (C_v)$ with $ \mathcal S_{\text{ext}} (C_{v,j})$ for some $j \in \{1, \dots, N\}$. If $v$ is a leaf, then we set $\mathcal E_{\text{int}}(C_v)=\emptyset$.
\end{itemize}\ \smallskip

\begin{remark} Note that if $v \in T$, and $C_v \in \mathcal B^{-}$, then $E(C_v)\backslash C_v$ is the set of positive cubes that are in the bounded component of $ \mathcal S_{\text{ext}} (C_v)$  and touch $\mathcal S_{\text{ext}}(C_v)$.  Also, if a negative cube in $\R^3 \backslash C_v$ is touching $C_v$, then it does so through an edge in $ \mathcal E_{\text{ext}} (C_v)$. \\ \smallskip
\end{remark}

 \begin{center}
\includegraphics[height=5cm]{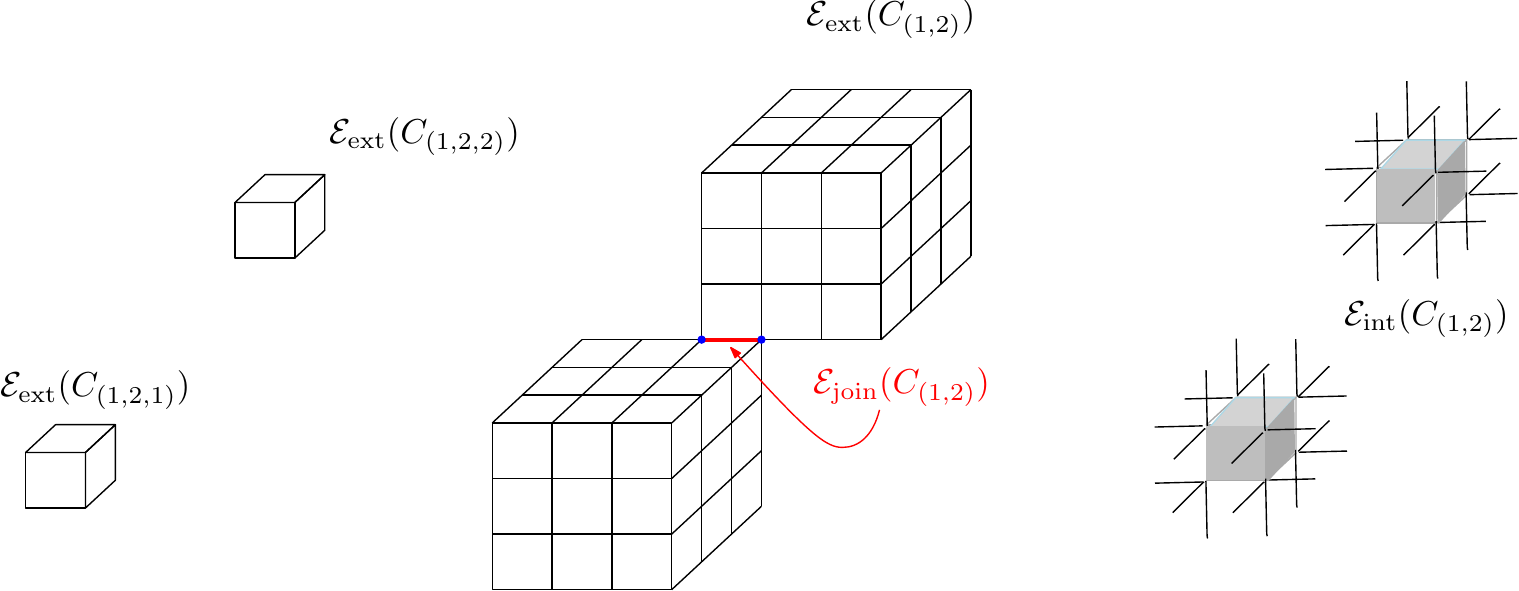}
 \end{center}\ \\

\begin{remark}  Given a node $v$ with children $\{(v,j)\}_{j=1}^N$, let $\mathcal G(C_{(v,j)})$ be the set of edges in $\{x \in \R^3:\; d(x, C_{(v,j)})=1\}$. It is clear that for each $j=1, \dots, N$ the set  $\mathcal G(C_{(v,j)})$ is connected. Also, $\mathcal E_{\text{ext}} (C_v) = \cup_{j=1}^N \mathcal G(C_{(v,j)}) \backslash \mathcal E_{\text{join}}(C_v)$. Since the  edges in $\mathcal E_{\text{join}}(C_v)$ are open, the structures  $\mathcal E_{\text{ext}} (C_v)$ are connected.
\end{remark}

\noindent  We  proceed to define a perturbation  $h:K \to \R$, where 
$$K=\bigcup_{v \in T} \mathcal E_{\text{ext}}(C_v)\cup \mathcal E_{\text{int}}(C_v)\cup   \mathcal E_{\text{join}}(C_v).$$ 
We note that by construction $K$ is formed by all the edges in $C_\emptyset$. 
Also, it is important to note that if two adjacent cubes have the same sign, then they share an edge in $K$.
The function $h$ is defined by the rules $A$, $B$ and $C$ below. \\

\begin{itemize}
\item[A)]  \emph{Perturbation on $\mathcal E_{\text{ext}} (C_v)$.} Let $v \in T$ and assume $C_v \in \mathcal B^-$. We define $h$  on every edge of  $\mathcal E_{\text{ext}} (C_v)$  to be $1$.   If  $C_v  \in \mathcal B^+$, we define $h$  on every edge of  $\mathcal E_{\text{ext}} (C_v)$  to be $-1$.   \\
\end{itemize}

\noindent Rule $A$ is meant to separate $C_v$ from all the exterior cubes of the same sign that surround it. Note that for all $v \in T$ we have $ \mathcal E_{\text{ext}} (C_v) \cap  \mathcal E_{\text{ext}} (C_{(v,j)})=\emptyset$, where $(v,j)$ is any of the children of $v$, so Rule A is well defined.\\

\begin{itemize}
\item[B)]  \emph{Perturbation on $\mathcal E_{\text{int}} (C_v)$.}  Let $e$ be an edge in $\mathcal E_{\text{int}}(C_v)$.  Then, we already know that $h$ is $1$ on one vertex and $-1$ on the other vertex. 
We extend $h$ smoothly to the entire edge $e$ so that it has a unique zero at the midpoint of $e$, and so that the absolute value of the derivative of $h$ is  $\geq 1$. We also ask for the derivative of $h$ to be $0$ at the vertices. For example,   if the edge is  $\{(a,b,z):\, z\in [0,1]\}$ where $a,b,c \in \Z$, we could take
 $h(a,b,z)=\cos(\pi z)$.  \\
\end{itemize}

\noindent Rule $B$ is enforced to ensure that no holes are added between edges that join a structure $C_v$ with any of its children structures $C_{(v,j)}$.\\

Next, assume $C_v \mathcal \in B^-$. Note that for any edge $e$ in  $\mathcal E_{\text{join}}(C_v)$ we  have that the function $h$ takes the value $1$ at their vertices, since those vertices belong to edges in $ \mathcal E_{\text{ext}}(C_v)$ and the function $h$ is defined to be $1$ on $ \mathcal E_{\text{ext}}(C_v)$. We have the same picture if  $C_v \in \mathcal B^+$, only that $h$ takes the value $-1$ on the vertices of all the joining edges. We therefore extend $h$  to be defined on $e$ as follows.\\

\begin{itemize}
\item[C)]   \emph{Perturbation on $\mathcal E_{\text{join}} (C_v)$.} Let $v \in T$ and assume $C_v \in \mathcal B^-$.  Given an edge in  $\mathcal E_{\text{join}}(C_v)$ we already know that $h$ takes the value $1$ at the vertices of the edge. We extend $h$ smoothly to the entire edge so that it takes the value $-1$ at the midpoint of the edge, and so that  it only has two roots at which the absolute value of the derivative of $h$ is  $\geq 1$. 
We further ask $h$ to have zero derivative at the endpoints of the edge. For example, if the edge is  $\{(a,b,z):\, z\in [c,c+1]\}$ where $a,b,c \in \Z$, we could take
 $h(a,b,z)=\cos(2\pi z)$.  In the case in which $C_v \in \mathcal B^+$  we  need $h$ to take the value $+1$ at the midpoint of the edge.\\
\end{itemize}

\noindent Rule $C$ is meant to glue the structures $\{ E (C_{(v,j)})\}_{j=1}^N$ through the middle point of the edges that join them, without generating new holes.\\

\begin{remark}\label{extension}
By construction the function $h$ is smooth in the interior of each edge. Furthermore, since we ask the derivative of $h$ to vanish at the vertices in $K$, the function $h$ can be extended to a function $h \in C^1(\Omega)$ where $\Omega \subset \R^3$ is an open neighborhood of $K$.
\end{remark}

 \begin{definition}\label{perturbation}
 Given a tree $T$,  let $h \in C^1(\Omega)$ be defined following Rules A, B and C and Remark \ref{extension}, where $\Omega \subset \R^3$ is an open neighborhood of $K$.  Since $K$ is compact and $\R^3 \backslash K$ is connected,  Theorem \ref{T:perturbation} gives the existence of  $f: \R^3 \to \R$  that satisfies 
 \[-\Delta f=f \quad \text{and} \quad  \sup_{ K}\{|f-h|+ \|\nabla f-\nabla h\|\} \leq \tfrac{1}{100}.\] 
For $\ep>0$ small set 
 \[u_\ep:= u_0 +\ep f.\]  
 \end{definition}
 We will show in Lemma \ref{limit} that the perturbation was built so that the nodal domain of $u_\ep$ corresponding to $v \in T$ is constituted by the deformed cubes in  $\bigcup_{j=1}^N E(C_{(v,j)} ) \backslash C_{(v,j)}$ after the perturbation is performed.

We illustrate how Rules A, B, and C work in the following examples.  In what follows we shall use repeatedly that the singularities of the zero set of $u_0$ are on the edges and vertices of the cubes. Therefore, the changes of topology in the zero set can only occur after perturbing the function $u_0$ along the edges and vertices of the cubes. \\

\noindent {\bf Example 1.} As an example of how  Rules A and B work, we explain how to create a domain that contains another nodal domain inside of it.  The tree corresponding to this picture is given by two nodes, $1$ and $(1,1)$, that are joined by an edge. 
We start with  a positive cube $C_{(1,1)} \in \mathcal B^+$ and work with its engulfment $C_1=E(C_{(1,1)}) \in \mathcal B^-$. All the edges of $C_{(1,1)}$ belong to  $\mathcal E_{\text{ext}} (C_{(1,1)})$. Therefore, the function $u_\ep$ takes the value $-\ep$ on $\mathcal E_{\text{ext}} (C_{(1,1)})$. Also,  all the positive cubes that touch $C_{(1,1)}$ do so through an edge in $\mathcal E_{\text{ext}} (C_{(1,1)})$. It follows that all the positive cubes surrounding $C_{(1,1)}$ are disconnected from $C_{(1,1)}$  after the perturbation is performed.  The cube $C_{(1,1)}$ then becomes a positive nodal domain $\Omega_{(1,1)}$ of $u_\ep$ that is contractible to a point.

 \begin{center}
\includegraphics[height=4.5cm]{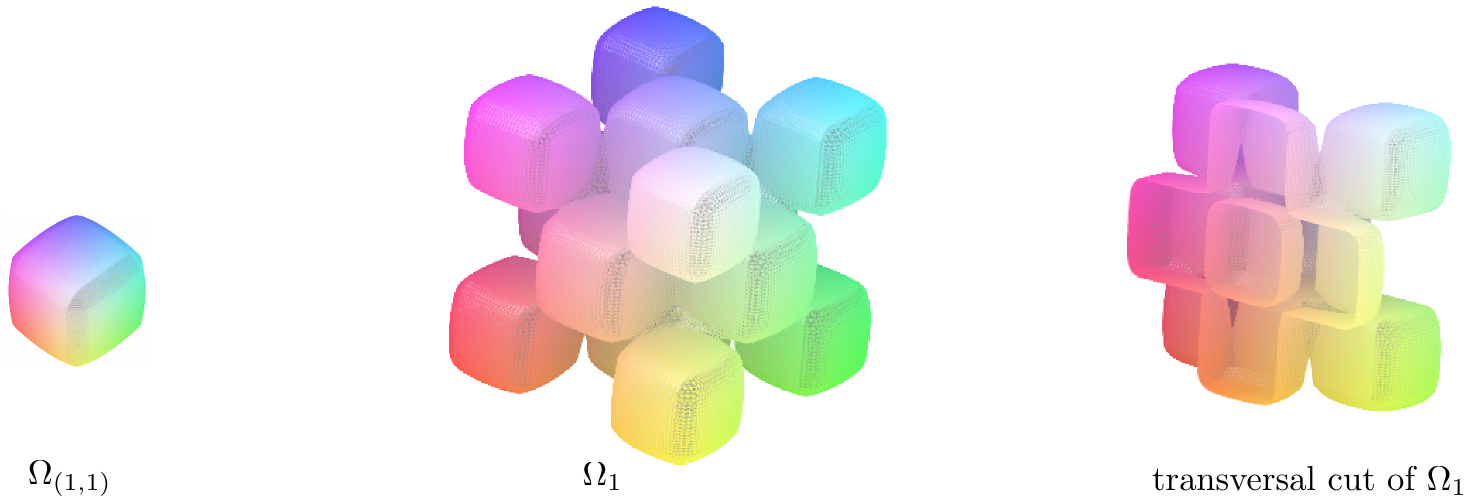}
\end{center}
 Next, note that all the negative cubes that touch $C_{(1,1)}$  (i.e., cubes in $E(C_{(1,1)}) \backslash C_{(1,1)}$) do so through a face whose edges are in   $\mathcal E_{\text{ext}} (C_{(1,1)})$, or through a vertex that also belongs to one of the edges in  $\mathcal E_{\text{ext}} (C_{(1,1)})$. Therefore, all the negative cubes are glued together after the perturbation is performed, and belong to a nodal domain $\Omega_1$ that contains the connected set  $\mathcal E_{\text{ext}} (C_{(1,1)})$.

So far we have seen that $\Omega_1$ contains  the perturbation of  the cubes in $E(C_{(1,1)}) \backslash C_{(1,1)}$. We claim that no other cubes are added to $\Omega_1$. Indeed, all the negative cubes that touch the boundary of  $E(C_{(1,1)}) =C_1$ do so through edges in  $\mathcal E_{\text{ext}} (C_{1})$. Then, since $u_\ep$ takes the value $\ep$ on  $\mathcal E_{\text{ext}} (C_{1})$, all the surrounding negative cubes are disconnected from $E(C_{(1,1)})$ after we apply the perturbation. Since along the edges connecting $\partial C_{(1,1)}$ with $\partial C_{1}$ the function $u_\ep$ has only one sign change (it goes from $-\ep$ to $\ep$) it is clear that $\Omega_1$ can be retracted to $\partial \Omega_{(1,1)}$.\\

\noindent {\bf Example 2.} Here we explain how Rule C works. Suppose we want to create a nodal domain that contains two disjoint nodal domains inside of it. The tree corresponding to this picture is given by three nodes, $1$, $(1,1)$, and $(1,2)$. The node $1$ is joined by an edge to  $(1,1)$ and by another edge to $(1,2)$.  Assume that $C_{(1,1)}$ and $C_{(1,2)}$ belong to $\mathcal B^+$.
Then, $C_1=E(C_{(1,1)}) \cup E(C_{(1,2)}) \in \mathcal B^-$.  When each of the structures $E(C_{(1,1)}) $ or $E(C_{(1,1)}) $ are perturbed, we get a copy of the negative nodal domain in Example 1. Since in $C_1$ the structures $E(C_{(1,1)}) $ and  $E(C_{(1,1)}) $  are joined by an edge, the two copies of $\Omega_1$ will also be glued. The reason for this is that the function $u_\ep$ takes the value $-\ep$ in the middle point of the edge joining $E(C_{(1,1)}) $ and  $E(C_{(1,1)}) $. Therefore,  a small negative tube connects both structures.

 \begin{center}
\includegraphics[height=5.5cm]{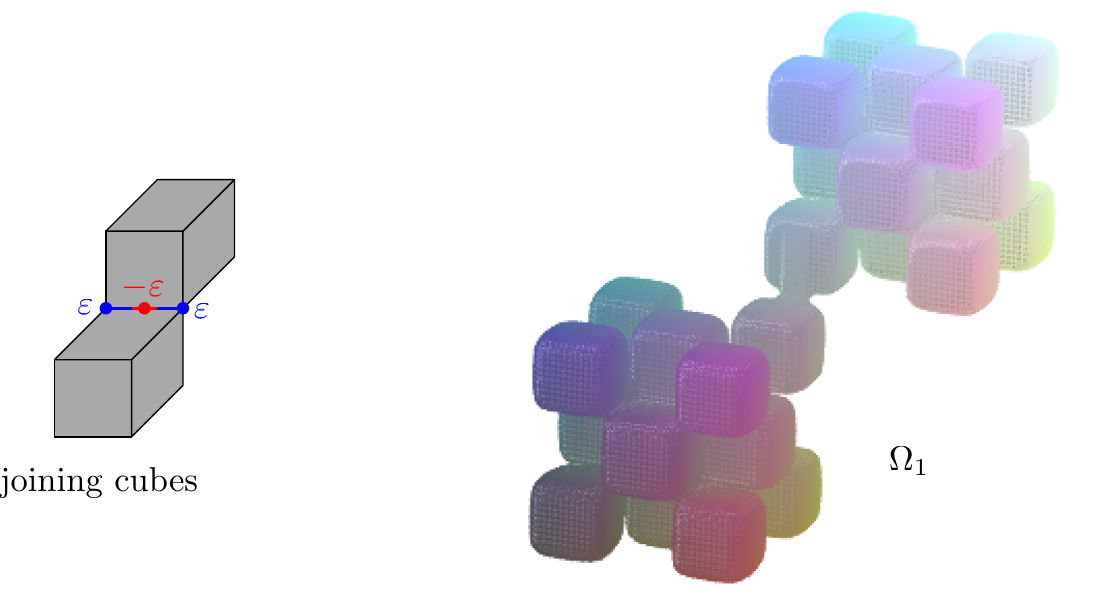}
 \end{center}

\subsection{Local behavior of the zero set}\label{local}
In this section we explain what our perturbation does to the zero set of $u_0$ at a local level. Given a tree $T$, and $\ep>0$,  let \[u_\ep = u_0 + \ep f\] be defined as in Definition \ref{perturbation}.
Using that $f$ is a continuous function, and that we are working on a compact region of $\R^n$ (we call it $D$), it is easy to see that   there exists a  $\delta_0>0$, so that if $\mathcal T_\delta$ is the $\delta$-tubular neighborhood  of $K$, then  $u_\ep$ has no zeros in $\mathcal T_\delta^c \cap C_\emptyset$ as long as $\delta \leq \delta_0$ and  
$$\ep =c_1 \delta^2,$$ where $c_1$ is some positive constant that depends only on $\|f\|_{C^0(D)}$. This follows after noticing that $|u_0|$ takes the value $1$ at the center of each cube and decreases radially until it takes the value $0$ on the boundary of the cube.   

The construction of the tubular neighborhood $\mathcal T_\delta$ yields that in order to understand the behavior of the zero set of $u_\ep$ we may restrict ourselves to study it inside $\mathcal T_\delta$ for $\delta \leq \delta_0$.
We proceed to study the zero set of $u_\ep$  in a $\delta$-tubular neighborhood of each edge in $K$. Assume, without loss of generality, that the edge is the set of points  $\{(0,0,z):\; z\in [0,1]\}$. \\

\noindent{\bf Vertices.} 
At the vertex $(0,0,0)$  the function $h$ takes the value $1$ or $-1$.  Assume $h(0,0,0)=-1$ (the study when the value is $1$ is identical).  In this case, we claim that the zero set of $u_\ep(x,y,z)$ near the vertex is diffeomorphic to that of the function  $\ell_\ep (x,y,z):=u_0(x,y,z)-\ep$ provided $\delta$ (and hence $\ep=\ep(\delta)$) is small enough.  To see this, for  $\eta>0$  set $V_\eta$ to be one of the connected components of $u_\ep^{-1}(B(0, \eta))$ intersected with $\mathcal T_\delta$. 

We apply the version of Thom's Isotopy Theorem given in [EP, Theorem 3.1] which asserts  that for every smooth function $\ell$ satisfying 
 \begin{equation}\label{RHS}
 \|u_\ep -\ell \|_{C^1(V_\eta)}\leq  \min\big\{\eta/4, \; 1\;, \inf_{V_\eta} \|\nabla u_\ep\| \big \}
 \end{equation}
   there exists a diffeomorphism $\Phi:\R^3 \to \R^3$ making 
$$\Phi(u_\ep^{-1}(0) \cap V_\eta) = \ell^{-1}(0) \cap V_\eta.$$ 
    We observe that the statement of  [EP, Theorem 3.1]  gives the existence of an $\alpha>0$ so that   the diffeomorphism can be built provided $\|\ell_\ep-u\|_{C^1(V_\eta)}\leq \alpha$. However, it can be tracked from the proof that $\alpha$ can be chosen to be as in the RHS of \eqref{RHS}. 
    
  Applying [EP, Theorem 3.1]  to the function $\ell_\ep$ we obtain what we claim provided we can verify \eqref{RHS}.  
  First, note that $\|u_\ep-\ell_\ep\|_{C^1(V_\eta)}=\|\ep(f-1)\|_{C^1(V_\eta)}$. It  is then easy to check that 
  \begin{equation}\label{bound on diff}
  \|u_\ep-\ell_\ep\|_{C^1(V_\eta)}\leq c_2 \ep
  \end{equation}
 for some $c_2>0$ depending only on $\|\nabla f\|_{C^0(D)}$.
Next, we find a lower bound for the gradient of $u_\ep$ when restricted to the zero set $u_\ep^{-1}(0)$. Note that for $(x,y,z) \in \mathcal T_\delta \cap u_\ep^{-1}(0)$ we have
\begin{align}
\|\nabla u_\ep (x,y,z)\|
&= \ep \Big  \|  -\pi  f(x,y,z) \big( \cot(\pi x), \cot (\pi y), \cot(\pi z) \big)    + \nabla f(x,y,z) \Big \|   \label{gradient}   \\
& \geq \ep \Big(\sqrt{ \frac{1}{x^2}+\frac{1}{y^2}+\frac{1}{z^2}}- \|\nabla f(x,y,z)\|\Big) + O(\ep \delta) \notag\\
& \geq \ep \Big(\frac{1}{\delta}- \|\nabla f(x,y,z)\|\Big) + O(\ep \delta).\notag
\end{align}
On the other hand, since $ \|\nabla \langle \text{Hess} \,u_\ep (x,y,z) , (x,y,z) \rangle\|=O(\eta)$ for all $(x,y,z) \in  V_\eta$, we conclude
\begin{equation}\label{bound on grad}
\inf_{V_\eta} \|\nabla u_\ep \| > \ep\Big(\frac{1}{\delta}- \|\nabla f(x,y,z)\|\Big) + O(\ep \delta)+O(\eta)
\end{equation}
 whenever $\delta$ is small enough. 

Using the bounds in \eqref{bound on diff} and \eqref{bound on grad} it is immediate to check that  \eqref{RHS} holds  provided we choose 
 $\eta =  c_3 \ep$ for a constant  $c_3>0$  depending only on $f$, and for $\delta$ small enough.

In the image below the first figure shows the zero set of $u_0$ near $0$. The other two figures are of the zero set of $\ell_\ep(x,y,z)$.
\begin{center}
\includegraphics[height=4cm]{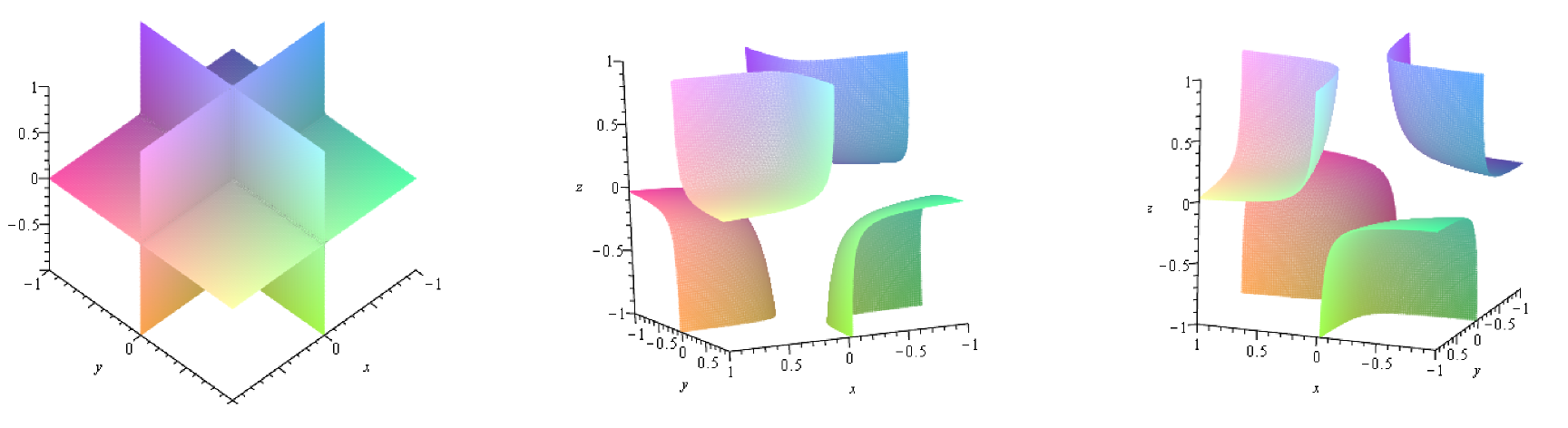}
\end{center}
This shows that at each vertex where $h$ takes the value $-1$ the negative cubes that touch the vertex are glued together while the positive ones are disconnected.\\

\noindent{\bf Edges.} Having dealt with the vertices we move to describe the zero set of the perturbation near a point inside the edge. 
There are three cases. In the first case (case A) the perturbation $h$ is strictly positive (approx. $\ep$) or strictly negative (approx $-\ep$) along the edge. In the second case (case B) the perturbation $f$ is strictly positive (approx. $\ep$)  at one vertex and strictly negative (approx. $-\ep$) at the other vertex. In the third case (case C), the edge is joining two adjacent structures so the perturbation $f$ takes the same sign at the vertices ( it is approx. $\ep$) and the opposite sign  (it is approx. $\ep$) at the midpoint of the edge having only two zeros along the edge.\\

In {\bf case A} the zero set of $u_\ep(x,y,z)$  near the edge is diffeomorphic to the zero set of the map $\ell_\ep(x,y,z):= u_0(x,y,z)-\ep$. The proof of this claim is the same as  the one given near the vertices, so we omit it. In the picture below the first figure shows the zero set of $u_0$ near the edge while the second figure shows the zero set of $\ell_\ep$. 
   \begin{center}
\includegraphics[height=4cm]{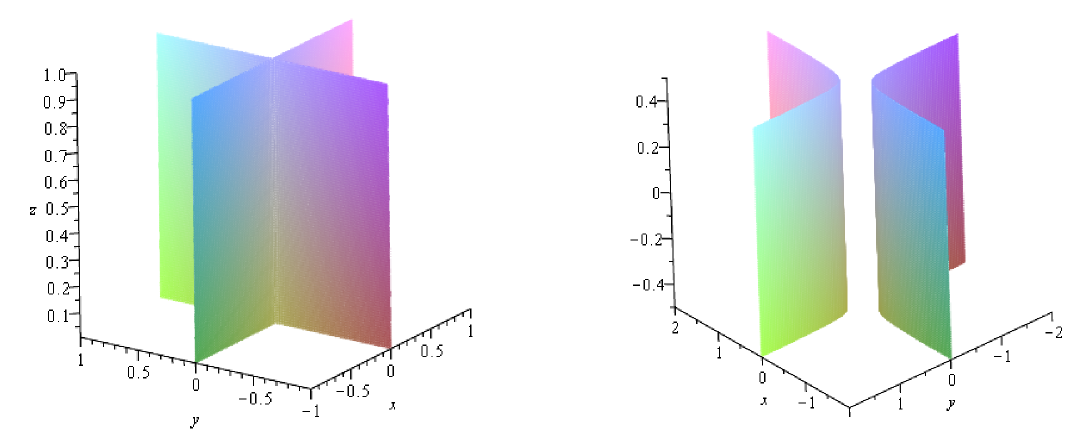}
 \end{center}
 This shows that two cubes of the same sign, say negative, that are connected through an edge are going to be either glued if the perturbation takes the value $-1$ along the edge, or disconnected if the perturbation takes the value $+1$ along the edge.\\

In {\bf case B}, it is clear that the only interesting new behavior will occur near the points on the edge at which the function $f$ vanishes.  Since $\|h-f\|_{C^1(\Omega)}< \frac{1}{100}$ and $h(0,0,b)=0$, there is only one point at which $f$ vanishes; say the point is $(0,0,b)$. Note that $f$ was built so that $(0,0,b)$ is the only zero of $f$ along the edge. We claim that the  zero set of $u_\ep$ near $(0,0,b)$  is diffeomorphic to the zero set of the map $\ell_\ep(x,y,z):=u_0(x,y,z)-\ep f(0,0,z)$.  The proof of this claim is similar to the one given near the vertices, so we omit it.  The only  relevant difference is that in order to bound $\|\nabla u_\ep\|$ from below, one uses that $\|\nabla u_\ep (x,y,z)\| \geq \|\ep \nabla f (x,y,z)\|- \|\nabla u_0 (x,y,z)\|$, and that $\|\nabla u_0 (x,y,z)\| =O(\beta)$ in a ball of radius $\beta$ centered at $(0,0,b)$ while  $\|\nabla f(0,0, b)\| > 1-\tfrac{1}{100}$. Of course, if one is away from the value $z=b$, then the analysis is the same as that of case A. The first figure in the picture below shows the zero set of $u_0$ along the edge while the second figure shows the zero set of $\ell_\ep$ when  $f(0,0,z)=\cos(\pi z)$.
  \begin{center}
\includegraphics[height=3.7cm]{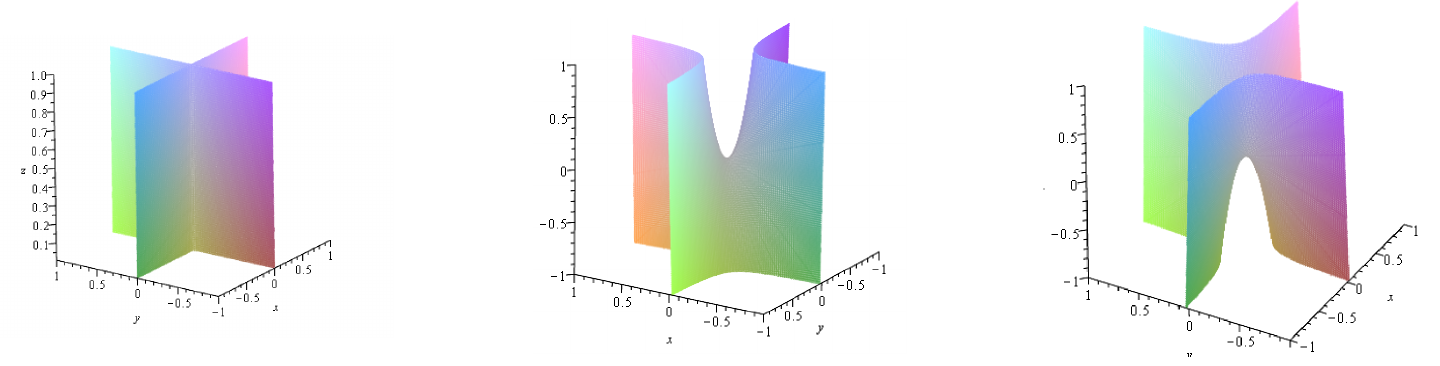}
 \end{center}
 This shows that two consecutive cubes sharing an edge along which the perturbation changes sign will be glued on one half of the edge and disconnected along the other half.\\
 
 In {\bf case C}, the zero set of $u_\ep$ is diffeomorphic to that of $\ell_\ep(x,y,z)=u_0(x,y,z)+\ep f(0,0, z)$
 where $f$ satisfies $\|h-f\|_{C^1(\Omega)}< \frac{1}{100}$ and $h(0,0,0)=h(0,0,1)=1$ and $h(0,0,1/2)=-1$.
  The zero set of $\ell_\ep$ when $ f(0,0,z)=\cos(2 \pi z)$ is plotted in the figure below.
  \begin{center}
\includegraphics[height=3cm]{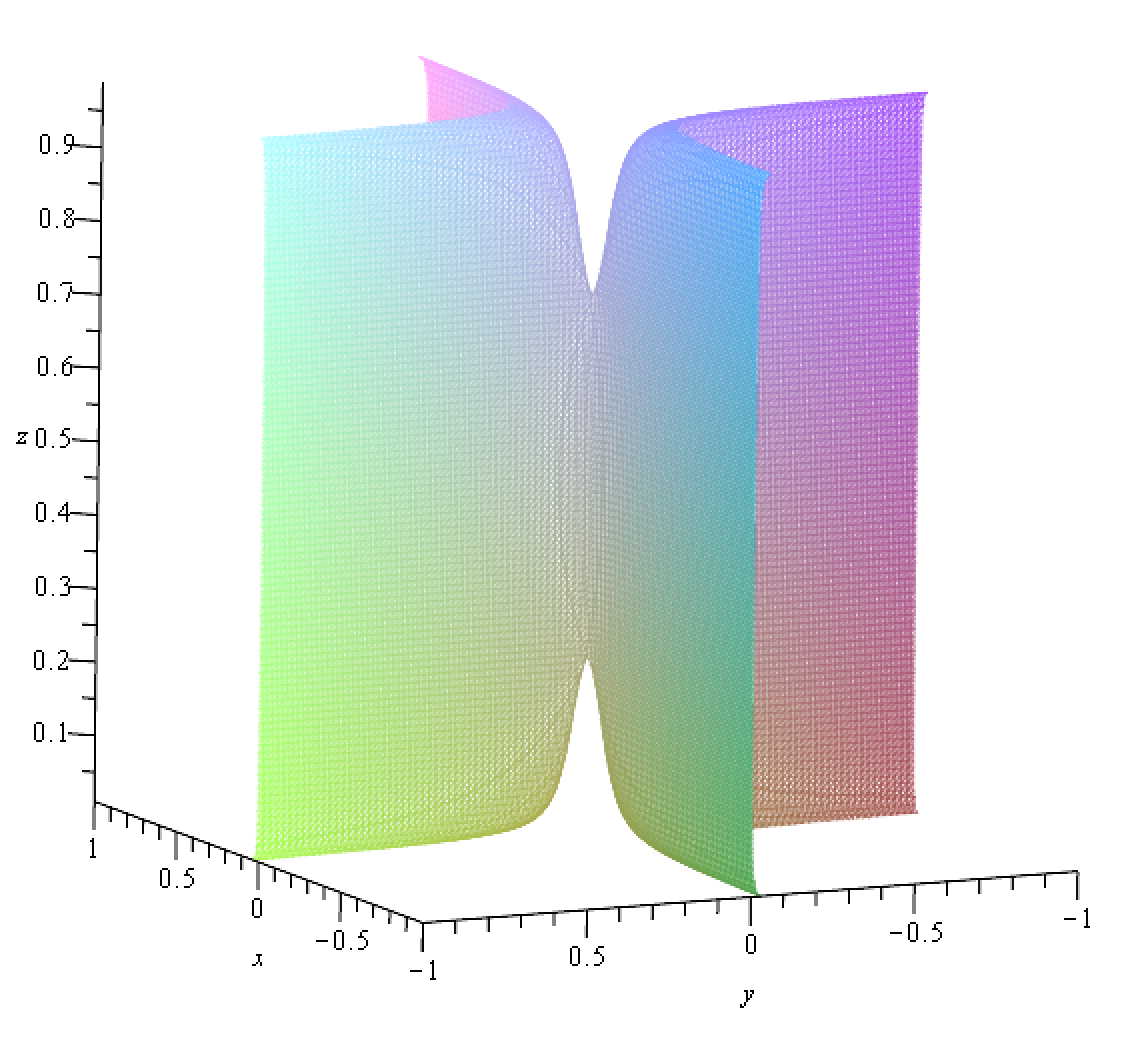}\quad\qquad\quad
\includegraphics[height=3cm]{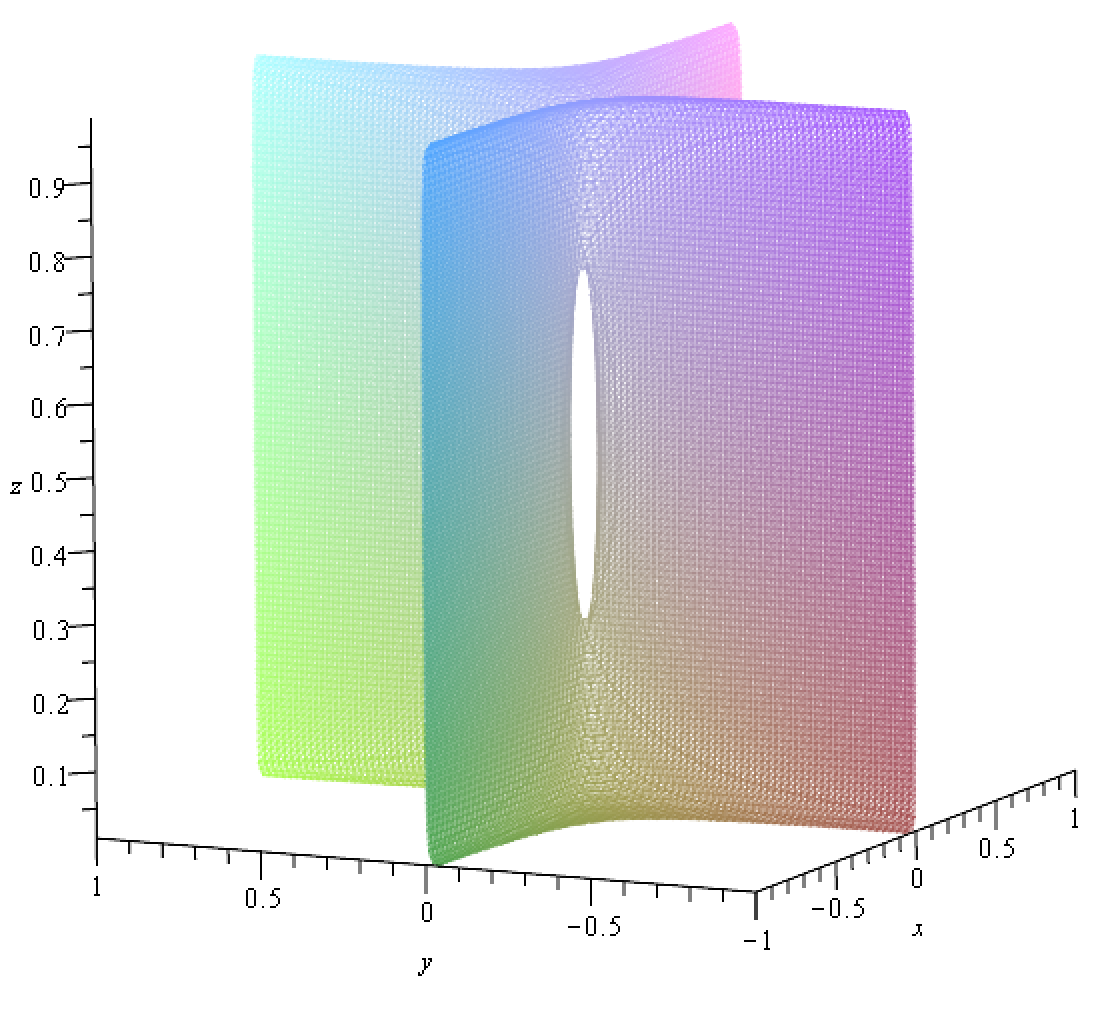}\quad\qquad\quad
\includegraphics[height=3cm]{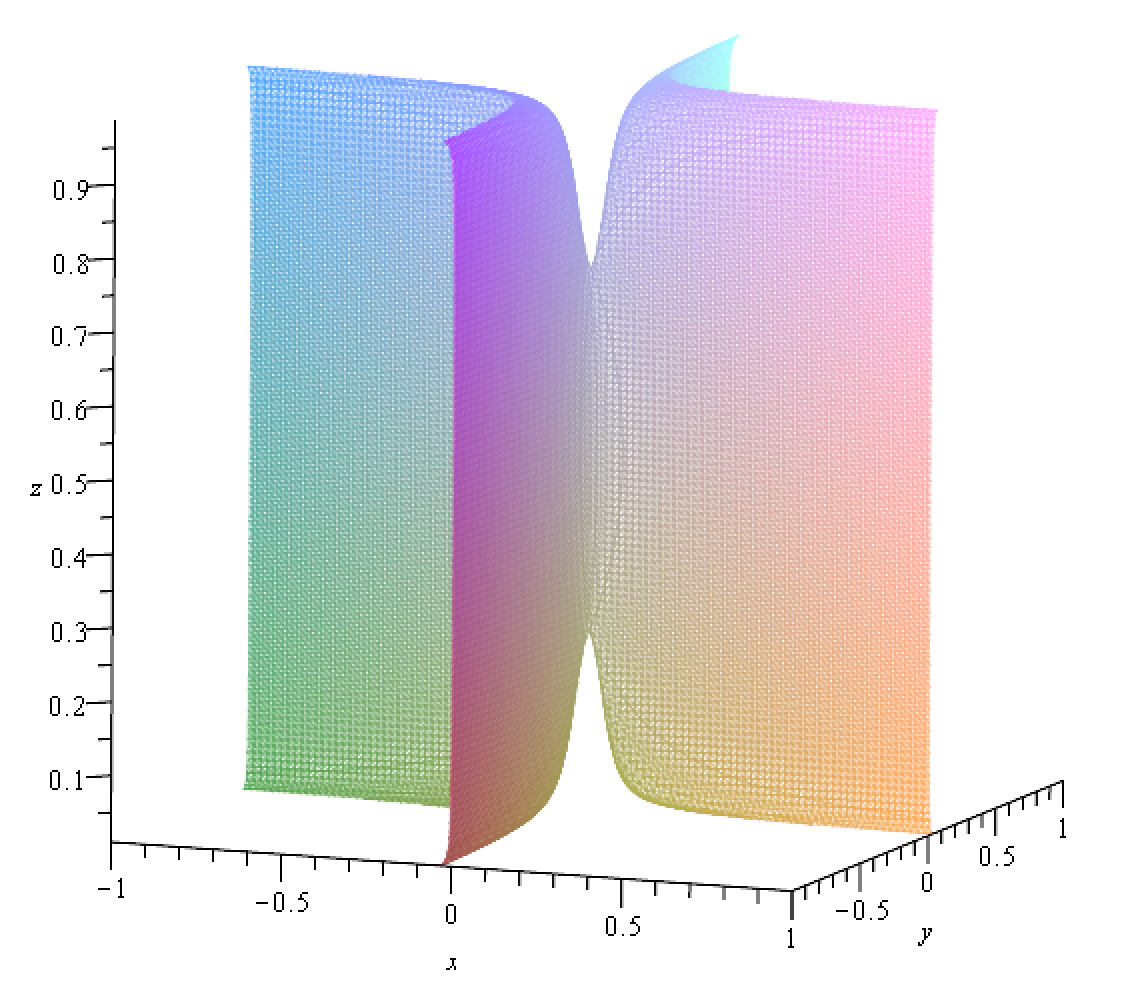}
 \end{center}
 This shows that two cubes that are joining two consecutive structures will be glued though the midpoint while being disconnected at the vertices.
 
\subsection{Definition of the nodal domains}
Given a tree $T$ and $\ep>0$  we continue to work with 
 \begin{equation*}
 u_\ep= u_0 +\ep f,
 \end{equation*} 
 as defined in Definition \ref{perturbation}.
 Fix $v \in T$, and suppose it has $N$ children. Assume without loss of generality that $C_v \in \mathcal B^+$.  For every $j \in \{1, \dots, n\}$ the perturbed function  $u_\ep$ takes the value  $\ep$ on $\mathcal E_{\text{ext}} (C_{(v,j)})$, and $\mathcal E_{\text{ext}} (C_{(v,j)})$ is connected.  It follows that for each $j \in\{1, \dots, N\}$ there exists a positive nodal domain $\mathcal N_{(v,j)}$  of $u_\ep$ that contains  $\mathcal E_{\text{ext}} (C_{(v,j)})$. We define the set $\Omega_{v}=\Omega_{v}(\ep)$ as 
 \begin{equation}\label{domain}
 \Omega_{v}:=\bigcup_{j=1}^N \mathcal N_{(v,j)}.
 \end{equation}
 Throughout this section we  use  the description of the local behavior of $u_\ep^{-1}(0)$ that we gave in Section \ref{local}.
In the following lemma we prove that  $\Omega_v$ is a nodal domain of $u_\ep$.

\begin{lemma}\label{nodal d}
Let $T$ be a tree and for each  $\ep>0$ let $u_\ep$ be the perturbation defined in \eqref{perturbation}.  Then, for each $\ep>0$ and $v \in T$, the set $\Omega_v=\Omega_v(\ep)$  defined in \eqref{domain} is a nodal domain of $u_\ep$.
\end{lemma}

\begin{proof} 
Let $v \in T$ and suppose $v$ has $N$ children. Assume without loss of generality that $C_v \in \mathcal B^-$. By definition, $\Omega_{v}=\cup_{j=1}^N \mathcal N_{(v,j)}$ where  $\mathcal N_{(v,j)}$  is the nodal domain of $u_\ep$ that contains  $\mathcal E_{\text{ext}} (C_{(v,j)})$. To prove that $\Omega_v$ is itself a nodal domain, we shall show that $\mathcal N_{(v,j)}=\mathcal N_{(v,j+1)}$ for all $j \in \{1, \dots, N-1\}$.

Fix $j \in \{1, \dots, N-1\}$. The structures $E(C_{(v,j)})$ and $E(C_{(v,j+1)})$ are joined through an edge $e_j$ in $\mathcal E_{\text{join}}(C_v)$. If we name the middle point of $e_j$ as $m_j$,  then by Rule C we have  $u_\ep(m_j)=\ep f(m_j)<0$. 

The edge $e_j$ is shared by a cube $c_j \in E(C_{(v,j)})$ and a cube $c_{j+1} \in E(C_{(v,j+1)})$.
Note that every cube in $E(C_{(v,j)})$ has at least one vertex that belongs to an edge in $\mathcal E_{\text{ext}} (C_{(v,j)})$ (same with  $E(C_{(v,j+1)})$). Let  $p_j$ be a vertex of  $c_j$ that belongs to an edge in  $\mathcal E_{\text{ext}} (C_{(v,j)})$. In the same way we choose  $q_{j}$ to be a vertex in $ c_{j+1}$ that belongs to an edge in  $\mathcal E_{\text{ext}} (C_{(v,j+1)})$. In particular, by Rule~A we have that  $u_\ep(p_j)<0$ and $u_\ep(q_j)<0$.

 \begin{figure}[h!]
\includegraphics[height=6cm]{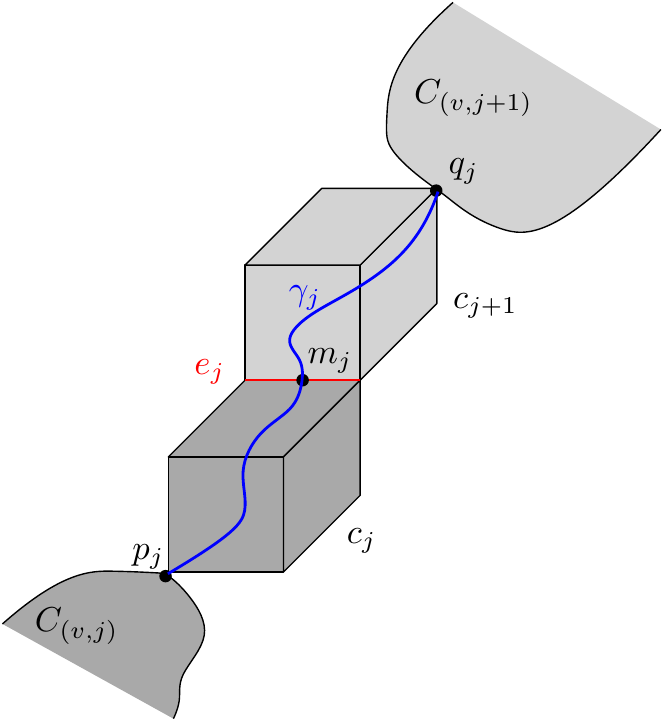}
\caption{}
\label{figure join}
 \end{figure}
Since both $c_j$ and $c_{j+1}$ are negative cubes, there exists a curve $\gamma_j \subset u_\ep^{-1}((-\infty,0))$ that joins $p_j $ with $q_j$  while passing through the middle point $m_j$. 

Finally, since $p_j \in \mathcal E_{\text{ext}} (C_{(v,j)}) \subset  \mathcal N_{(v,j)}$,  $q_j \in \mathcal E_{\text{ext}} (C_{(v,j+1)}) \subset  \mathcal N_{(v,j+1)} $, and $\gamma_j$ is a connected subset of $u_\ep^{-1}((-\infty, 0))$, we must have that $\mathcal N_{(v,j)}=\mathcal N_{(v,j+1)}$ as claimed.

\end{proof}
 
In the following lemma we describe the set of cubes that end up building a nodal domain after the perturbation is performed.

\begin{lemma}\label{limit}
Let $T$ be a tree and for each  $\ep>0$ let $u_\ep$ be the perturbation defined in \eqref{perturbation}.  For each  $v \in T$ with $N$ children we have 
\[\lim_{\ep \to 0} \Omega_v(\ep)=  \bigcup_{j=1}^N E(C_{(v,j)} ) \backslash C_{(v,j)}.\]
\end{lemma}

\begin{proof}
First, we show that all the cubes in $\cup_{j=1}^N E(C_{(v,j)} ) \backslash C_{(v,j)}$ glue to form part of $\Omega_v$ after the perturbation is performed.  Assume, without loss of generality, that $C_v \in \mathcal B^+$. Then,  $C_{(v,j)} \in \mathcal B^-$ for every child $(v,j)$ of $v$.
All the cubes in $\cup_{j=1}^N E(C_{(v,j)} ) \backslash C_{(v,j)}$ have an edge in   $\mathcal E_{\text{ext}} (C_{(v,j)})$. Since such cubes are positive, and $u_\ep$ takes the value $\ep$ on $\mathcal E_{\text{ext}} (C_{(v,j)})$, it follows that  the cubes become part of the nodal domain that contains $\mathcal E_{\text{ext}} (C_{(v,j)})$. That is, all the cubes in $\cup_{j=1}^N E(C_{(v,j)} ) \backslash C_{(v,j)}$  become part of $\Omega_v$ after the perturbation is added to $u_0$.

Second, we show that no cubes, other than those in $ \cup_{j=1}^N E(C_{(v,j)} ) \backslash C_{(v,j)}$, will glue to  form part of $\Omega_v$. Indeed, any other positive cube in  $\R^3 \backslash \cup_{j=1}^N E(C_{(v,j)} ) $ that touches $\partial  ( \cup_{j=1}^N E(C_{(v,j)} ) )$ does so through an edge in $\mathcal E_{\text{ext}} (C_{v})$. Since the function $u_\ep$ takes the value $-\ep$ on $\mathcal E_{\text{ext}} (C_{v})$, those cubes will disconnect from $\cup_{j=1}^N E(C_{(v,j)} ) $ after we perturb. On the other hand, any positive cube in $ \cup_{j=1}^N C_{(v,j)} \in \mathcal B^-$ is touching $ \cup_{j=1}^N E(C_{(v,j)} )$ through edges in $ \cup_{i=1}^{N_j} \mathcal E_{\text{ext}} (C_{(v,j,i)})$ where $N_j$ is the number of children of $(v,j)$. Since $f$ takes the value $-\ep$ on $ \cup_{i=1}^{N_j} \mathcal E_{\text{ext}} (C_{(v,j,i)})$, the cubes in $ \cup_{j=1}^N C_{(v,j)} $ will also disconnect from $ \cup_{j=1}^N E(C_{(v,j)} ) \backslash C_{(v,j)}$.
\end{proof}

It is convenient to define the partial collections of nested domains. Given a tree $T$ , a perturbation  $u_\ep$, and $v \in T$, we define the collection $\Omega_v^*=\Omega_v^*(\ep)$ of all nodal domains that are descendants of $\Omega_v$ as follows. If $v$ is a leaf then $\Omega_v^*=\Omega_v$. If $v$ is not a leaf and has $N$ children, we set
 \[\Omega_v^*:= \overline{\Omega_v}\, \cup\, \bigcup_{j=1}^N \Omega_{(v,j)}^*. \]

\begin{remark} A direct consequence of Lemma \ref{limit} is the following.
Let $T$ be a tree and for each  $\ep>0$ let $u_\ep$ be the perturbation defined in \eqref{perturbation}.  For each  $v \in T$, 
\[\lim_{\ep \to 0} \Omega_v^*(\ep)= C_v.\]
\end{remark}

 \subsection{Proof of Theorem \ref{main theorem R b}}
We will use throughout this section that we know how the zero set behaves at a local scale (as described in Section \ref{local}).
 Let $T$ be a tree and for each  $\ep>0$ let $u_\ep$ be the perturbation defined in \eqref{perturbation}.  We shall prove that there is a subset of the nodal domains of $u_\ep$ that are nested as prescribed by $T$. Since for every $v \in T$ the set $\Omega_v$ is a nodal domain of $u_\ep$, the theorem would follow if we had that for all $v \in T$ \medskip
\begin{itemize} 
\item[(i)] $\Omega_{(v,j)}^* \subset \text{int} (\Omega_v^*)$  for every $(v,j)$ child of $v$.\\
\item[(ii)] $\Omega_{(v,j)}^* \cap \Omega_{(v,k)}^* =\emptyset$ for all $ j\neq k$.\\
\item[(iii)]  $\R^3 \backslash \Omega_v^*$ has no bounded component.\\
\end{itemize}

Statements (i), (ii) and (iii) imply that $\R^3 \backslash \Omega_v$ has $N+1$ components. One component is unbounded, and each of the other $N$ components is filled by  $\Omega_{(v,j)}^*$ for some $j$.
We prove statements (i), (ii) and (iii)  by induction. The statements are obvious for the leaves of the tree.\\

\begin{remark}
The proof of Claim (iii) actually shows that $\Omega_v$ can be retracted to the arc connected set $\bigcup_{j=1}^N \Omega_{(v,j)}^* \cup \bigcup_{j=1}^{N-1}\gamma_j$ where $\gamma_j \subset \Omega_v$ is the curve  introduced in Lemma \ref{nodal d} connecting  $ \mathcal E_{\text{ext}} (C_{(v,j)})$ with $\mathcal E_{\text{ext}} (C_{(v,j+1)})$  that passes through the midpoint of the edge joining  $E(C_{(v,j)})$ with $E(C_{(v,j+1)})$.
\end{remark}

 \emph{Proof of Claim (i)}. Since $\Omega_v^*= \overline{\Omega_v}\, \cup\, \bigcup_{j=1}^N \Omega_{(v,j)}^*$, we shall show that there exists an open neighborhood $\mathcal U_{(v,j)}$ of  $\Omega_{(v,j)}^*$ so that 
 $\mathcal U_{(v,j)} \subset \Omega_v^*.$
 
 Assume without loss of generality that $C_v \in \mathcal B^+$.  Then, for every child $(v,j)$, all the faces in $\partial C_{(v,j)}$ belong to cubes in $C_{(v,j)}$ that are negative. Also, all the other negative cubes in $\R^3 \backslash C_{(v,j)}$ that touch  $\partial C_{(v,j)}$ do so through an edge in $\mathcal E_{\text{ext}} (C_{(v,j)})$. Since the function $u_\ep$ takes the value $\ep$ on $\mathcal E_{\text{ext}} (C_{(v,j)})$, all the negative cubes in $ C_{(v,j)}$ are disconnected from those in $\R^3 \backslash C_{(v,j)}$ after the perturbation is performed. While all the negative cubes touching  $ C_{(v,j)}$  are disconnected, an open positive layer $\mathcal L_{(v,j)}$ that surrounds $\Omega^*_{(v,j)}$ is created. The layer $\mathcal L_{(v,j)}$ contains the grid $\mathcal E_{\text{ext}} (C_{(v,j)})$ and so it is  contained inside $\Omega_v$. The result follows from setting  $\mathcal U_{(v,j)}:=\mathcal L_{(v,j)} \cup \Omega^*_{(v,j)} $. \\
 
\emph{Proof of Claim (ii)}.  This is a consequence of how we proved the statement (i) since both  $\Omega_{(v,j)}^*$ and  $\Omega_{(v,k)}^*$ are surrounded by  a positive layer inside $\Omega_v^*$.  \ \\
 
 \emph{Proof of Claim (iii)}. Note that $\lim_{\ep \to 0} \cup_{j=1}^N  \Omega^*_{(v,j)}(\ep)=\cup_{j=1}^N C_{(v,j)}$ and that by the induction assuption $\R^3 \backslash \cup_{j=1}^N  \Omega^*_{(v,j)}$ has no bounded components . On the other hand, we also have that $\lim_{\ep \to 0} \Omega_v(\ep)= \cup_{j=1}^N E(C_{(v,j)})\backslash C_{(v,j)}.$ This shows that, in order to prove that $\R^3 \backslash \Omega_v^*$ has no bounded components, we should show that the cubes in $\cup_{j=1}^N E(C_{(v,j)})\backslash C_{(v,j)}$ glue to those in  $\cup_{j=1}^N C_{(v,j)}$ leaving no holes. 
    Note that all the cubes in $\cup_{j=1}^N E(C_{(v,j)})\backslash C_{(v,j)}$  are attached to  the mesh $\cup_{j=1}^N \mathcal E_{\text{ext}} (C_{(v,j)})$  through some faces or vertices.   
  
   Assume without loss of generality that   $C_v \in \mathcal B^+$. 
  For each $j \in \{1, \dots, N\}$ the layer $\mathcal L_{(v,j)}$ is contained in $\Omega_v$ and all the cubes in $E(C_v)\backslash C_v$ are glued to the layer thorugh an entire face or vertex. The topology of $\Omega_v$ will depend exclusively on how the cubes  in  $E(C_{(v,j)})\backslash C_{(v,j)}$  will join or disconnect each other along the edges that start at   $\mathcal E_{\text{ext}} (C_{(v,j)})$ and end at a distance $1$ from $\mathcal E_{\text{ext}} (C_{(v,j)})$.  
The function  $u_\ep$ takes the value $\ep$ on $\mathcal E_{\text{ext}} (C_{(v,j)})$. Also, note that if a pair of positive cubes in the unbounded component of  $\R^3 \backslash \mathcal L_{(v,j)}$ share an edge $e$ that starts at $ \mathcal E_{\text{ext}} (C_{(v,j)})$  and ends at a distance $1$ from it, then the end vertex belongs to $\mathcal E_{\text{ext}} (C_v)$, and the function $u_\ep$ takes the value $-\ep$ at this point. 
  \begin{figure}
\includegraphics[height=4cm]{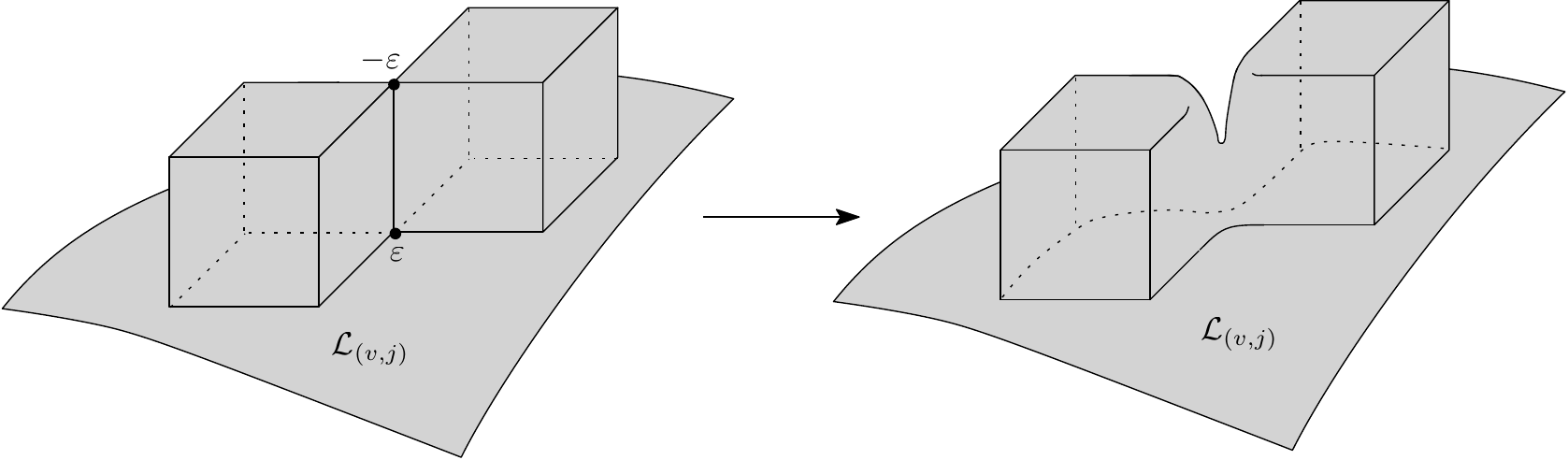}
 \end{figure}
Since the function $u_\ep$ has only one root on $e$, we have that no holes are added to $\Omega_v$ when applying the perturbation to those two cubes.    For cubes in the bounded component that share an edge one argues similarly and uses the value of $u_\ep$ on $ \cup_{i=1}^{N_j}\mathcal E_{\text{ext}} (C_{v,j,i})$ where $N_j$ is the number of children of $(v,j)$.

 To finish, we note that two consecutive  structures  $E(C_{(v,j)})$ and  $E(C_{(v,j+1)})$ are joined through an edge separating two  cubes as shown in Figure \ref{figure join}.  The function $u_\ep$ is negative (approximately equal to  $-\ep$) at the vertices of the edge, and  is positive at the middle point  (approximately equal to  $+\ep$). Since along the edge $u_\ep$ was prescribed to have only two roots, no holes are introduced when joining the structures.

\subsection{Higher dimensions}\label{higher dimensions}

The argument in higher dimensions is analogue to the one in dimension 3. We briefly discuss the modifications that need to be carried in this setting. 
Let \[u_0(x_1, \dots, x_n)=\sin(\pi x_1) \dots \sin(\pi x_n).\]
We will work with cubes in $\R^n$ that we identify with a point  $c \in  \mathbb Z^n$. That is, the cube corresponding to $c=(c_1, \dots, c_n)\in \Z^n$ is given by
$c=\{x \in \R^n: \;x_k \in [c_k, c_k+1]\}$. 
As before, we say that a cube is positive (resp. negative) if $u_0$ is positive (resp. negative) when restricted to it. 
 The collection of  faces of the cube $c$ is 
 $\cup_{1 \leq i \leq n} \cup_{ x_i \in \{c_i, c_i+1\}} \{x \in \R^n:  \;x_k \in [c_k, c_k+1]\;\; \forall k \neq i\}$. 
 The collection of edges is \[ \bigcup_{1 \leq i,j \leq n}\;\; \bigcup_{\substack{a_i\in \{c_i, c_i+1\}\\ a_j\in \{c_j, c_j+1\}}} H_c(a_i,a_j)\]
where each edge is described as the set
\[H_c(a_i,a_j)=\{x \in \R^n:\; x_i=a_i, \;\;\; x_j=a_j,\; \;\;x_k \in [c_k, c_k+1] \;\; \forall k \neq i,j\}.\]
We note that if two cubes of the same sign are adjacent, then they are connected through an edge or a subset of it.
In analogy with the $\R^3$ case, we define the collection $\mathcal B^+$ of all sets $\Omega$ that are built as a finite union of cubes with the following two properties:
\begin{itemize}
\item  $\R^n \backslash \Omega$ is connected. 
\item If $c$ is a cube in  $ \mathcal B^+$ with a face in $\partial B^+$, then $c$ must be a positive cube. 
\end{itemize}
We define  $\mathcal B^-$ in the same way only that the cubes with faces in $\partial \Omega$ should be negative cubes.

\ \\{\bf Engulf operation.} Let $C \in \mathcal B^+$. We define $E(C)$ to be the set obtained by adding to $C$ all the negative cubes that touch $C$, even if they share only one point with $C$. By construction   $E(C) \in \mathcal B^-$.  
If $C \in \mathcal B^-$, the set $E(C)$ is defined in the same form only that one adds positive cubes to $C$. In this case $E(C) \in \mathcal B^+$.

\ \\{\bf Join operation.} Given $C \in \mathcal B^+ \cup \mathcal B^-$ we distinguish two vertices using the lexicographic order.  For  $C \in \mathcal B^+ \cup \mathcal B^-$, let $\Gamma_C= C \cap \Z^n$ be the set of its vertices. We let $v_+(C)$ be the largest vertex in $\Gamma_C$ and $v_-(C)$ be the smallest vertex in $\Gamma_C$.
Given the vertex $v_+(C)$ we define the edge $e_+(C)$ to be the edge in $\partial C$ that contains the vertex $v_+(C)$ and is parallel to the hyperplane defined by the $x_1, \dots, x_{n-2}$ coordinates. The edge $e_-(C)$ is defined in the same way.

Given $C_1 \in \mathcal B^+$ and $C_2 \in \mathcal B^+$ we define $J(C_1, C_2) \in \mathcal B^+$  as follows. Let $\tilde C_2$ be the translated copy of $C_2$ for which $e_+(C_1)$ coincides with $e_-(\tilde C_2)$.
We ``join" $C_1$ and $C_2$  as
$J(C_1, C_2) = C_1 \cup \tilde C_2.$

In addition, for a single set $C$ we define $J(C)=C$, and if there are multiple sets $C_1, \dots, C_n$ we define $J(C_1, \dots, C_n)=J(C_1,  J(C_2, J(C_3, \dots J(C_{n-1},C_n)))).$

\ \\ {\bf Definition of the rough nested domains.} Given a tree $T$ we associate to each node  $v\in T$ a structure $C_{v} \subset \R^n$ defined as follows.  If the node   $v \in T$ is a leaf, then $C_{v}$ is a cube of the adequate sign. For the rest of the nodes we set
$ C_{ v}= J\left(E( C_{ (v,1)}), \dots, E( C_{ (v,N)}) \big.\right),$
where $N$ is the number of children of the node $v$.
We continue to identify the original structures $E( C_{ (v,j)})$ with the translated ones $\tilde E( C_{ (v,j)})$ that are used to build  $C_{v}$. After this identification, 
\[ C_{v}= \bigcup_{j=1}^N  E( C_{ (v,j)}).\]

\ \\ {\bf Building the perturbation.} 
Let  $v \in T$ be a node with $N$ children. We define the sets of edges $ \mathcal E_{\text{ext}}(C_v)$, $\mathcal E_{\text{int}}(C_v)$ and $\mathcal E_{\text{join}}(C_v)$ in exactly the same way as we did in $\R^3$ (see Section \ref{S: perturbation}).
  We  proceed to define a perturbation  $h:K \to \R$, where 
$$K=\bigcup_{v \in T} \mathcal E_{\text{ext}}(C_v)\cup \mathcal E_{\text{int}}(C_v)\cup   \mathcal E_{\text{join}}(C_v).$$ The function $h$ is defined by the rules $A$, $B$ and $C$ below. \\
Let $\chi:[0, \infty] \to [-1,1]$ be a smooth increasing function satisfying
\[\chi(0)=-1, \quad \chi(1/2)=0 \quad \text{and} \quad \chi(t) =1 \; \text{for}\; t \geq 1.\]
We also demand 
\begin{equation}\label{derivative}
\chi'(0)=0 \qquad \text{and}\qquad  \chi'(1/2) \geq 1.
\end{equation}\ \smallskip

\begin{itemize}
\item[A)] \emph{Perturbation on $\mathcal E_{\text{ext}} (C_v)$.} Let $v \in T$ and assume $C_v \in \mathcal B^-$. We define $h$  on every edge of  $\mathcal E_{\text{ext}} (C_v)$  to be $1$.   If  $C_v  \in \mathcal B^+$, we define $h$  on every edge of  $\mathcal E_{\text{ext}} (C_v)$  to be $-1$.     \\

\item[B)] \emph{Perturbation on $\mathcal E_{\text{int}} (C_v)$.} Let $H_c(a_i, a_j)$ be an edge that touches both $\mathcal E_{\text{ext}} (C_v)$ and $\mathcal E_{\text{ext}} (C_{(v,\ell)})$ for some of the child structures $C_{(v,\ell)}$ of $C_v$. Assume $C_v  \in \mathcal B^-$.  Then we know that we must have $h|_{\mathcal E_{\text{ext}} (C_v)}=1$ and $h|_{{\mathcal E_{\text{ext}} (C_{(v,\ell)})}}=-1$. Let $x_{i_1}, \dots, x_{i_k}$ be the set of directions in  $H_c(a_i, a_j)$ that connect $\mathcal E_{\text{ext}} (C_v)$ and $\mathcal E_{\text{ext}} (C_{(v,\ell)})$. 
 We let
\[h|_{H_c(a_i,a_j)}:H_c(a_i,a_j) \to [-1,1]\] be defined as
\[h(x_1, \dots, x_n)=\chi \left(\sqrt{ \sum_{m=1}^k (x_{i_m}-c_{i_m})^2}\right).\]
With this definition, since whenever $x \in \mathcal E_{\text{ext}} (C_{(v,\ell)})$ we have $x_{i_m}=c_{i_m}$ for all $m=1, \dots, k$, we get $h(x)=\chi(0)=-1$. Also, whenever $x \in \mathcal E_{\text{ext}} (C_{v})$ we have that there exists a coordinate $x_{i_m}$ for which $x_{i_m}=c_{i_m}+1$. Then, $\sum_m(x_{i_m}-c_{i_m})^2 \geq 1$ and so $h(x)=1$. Note that $h$ vanishes on the sphere $\mathcal S=\{x \in \R^n:\; \sum_{m=1}^k (x_{i_m}-c_{i_m})^2=1/4\}$ and that $\|\nabla h\| \geq 1$ on $\mathcal S$ because of \eqref{derivative}. If  $C_v  \in \mathcal B^+$, simply multiply $\chi$ by $-1$. \\

\item[C)] \emph{Perturbation on $\mathcal E_{\text{join}} (C_v)$.}  Let $v \in T$ and assume $C_v \in \mathcal B^-$. We set
\[h(x_1, \dots, x_n)=\chi \left( 2 \sqrt{\sum_{k=1}^{n-2} \big(x_{i_k} - \tfrac{2c_{i_k} +1}{2}\big)^2}\,\right),\]
where $i_k$ ranges over  the indices $\{1, \dots, n\}\backslash \{i,j\}$.
With this definition, whenever $x$ is at the center of the edge $H_c(a_i,a_j)$ we have $h(x)= \chi(0)=-1$. Also, if $x \in \partial H_c(a_i,a_j)$ we have $ \big(x_k - \tfrac{2c_k +1}{2}\big)^2 =1/4$ for some $k$, and so $h(x)=1$. Also note that $h$ vanishes on a sphere of radius $1/4$ centered at the midpoint of $H_c(a_i,a_j)$ and that the gradient of $h$ does not vanish on the sphere because of \eqref{derivative}.  If  $C_v  \in \mathcal B^+$, simply multiply $\chi$ by $-1$
\end{itemize}\ \\

\begin{remark}\label{extension2}
By construction the function $h$ is smooth in the interior of each edge. Furthermore, since according to \eqref{derivative} we have $\chi'(0)=0$ and $\chi'(1)=0 $, the  gradient of $h$ vanishes on the boundaries of the edges in $K$. Therefore,  the function $h$ can be extended to a function $h \in C^1(\Omega)$ where $\Omega \subset \R^n$ is an open neighborhood of $K$.
\end{remark}

 Given a tree $T$,  let $h \in C^1(\Omega)$ be defined following Rules A, B and C and Remark \ref{extension2}, where $\Omega \subset \R^n$ is an open neighborhood of $K$.  Since $K$ is compact and $\R^n \backslash K$ is connected,  Theorem \ref{T:perturbation} gives the existence of  $f: \R^n \to \R$  that satisfies 
 \[-\Delta f=f \quad \text{and} \quad  \sup_{ K}\{|f-h|+ \|\nabla f-\nabla h\|\} \leq \tfrac{1}{100}.\] 
For $\ep>0$ small set 
 \[u_\ep:= u_0 +\ep f.\]

The definitions in Rules A, B and C are the analogues to those in dimension $3$. For example, when working in dimension $3$ on the edge $e=\{(0,0,z):\; z\in [0,1]\}$, we could have set 
\[h(0,0,z)=\chi (z)  \;\; \qquad \text{if}\;\; e \in \mathcal E_{\text{int}} (C_v) \; \text{with} \;C_v \in  \mathcal B^- ,\]
and 
\[h(0,0,z)=\chi (2  |z-1/2|))  \;\; \qquad \text{if}\;\; e \in \mathcal E_{\text{join}} (C_v) \; \text{with}\; C_v \in  \mathcal B^-.\] \ \smallskip

Note that all the edges in $C_\emptyset$ are edges in $K$. Also, it is important to note that if two adjacent cubes have the same sign, then they share a subset of an edge in $K$.

If two adjacent cubes are connected through a subset of  $\mathcal E_{\text{ext}}(C_v)$, then the cubes will be either glued or separated along that subset. This is because the function $f$ is built to be strictly positive (approx. $\ep$) or strictly negative (approx. $-\ep$) along the entire edge.

If two adjacent cubes share an edge through which two structures are being joined, then they will be glued to each other near the midpoint of the edge. This is because $f$ is built so that it has the same sign as the cubes in an open neighborhood of the midpoint of the joining edge.

If two adjacent cubes in $C_v$ of the same sign share a subset of an edge in $H_c(a_i,a_j)\in \mathcal E_{\text{int}}(C_v)$, then with the same notation as in Rule B, there exists a subset of directions $\{x_{i_{m_1}}, \dots x_{i_{m_s}}\}\subset  \{x_{i_1}, \dots, x_{i_k}\}$ so that the set $\mathcal R=\{x \in H_c(a_i,a_j):\;\;x_{i_{m_t}} \in [c_{i_{m_t}}, c_{i_{m_t}}+1]  \; \forall t=1, \dots, s \}$ is shared by the cubes. By construction, the cubes will be glued through the portion $\mathcal R_1$ of $\mathcal R$ that joins $(c_{i_{m_1}}, \dots, c_{i_{m_s}})$ with the point $(z_1, \dots, z_s)$ near the midpoint $\big(c_{i_{m_1}}+\tfrac{1}{2}, \dots,c_{i_{m_s}}+ \tfrac{1}{2}\big)$, while being disconnected through the portion $\mathcal R_2$ of $\mathcal R$ that joins  the point $(z_1, \dots, z_s)$ with $(c_{i_{m_1}}+1, \dots, c_{i_{m_s}}+1)$. This is because $f$ is prescribed to have the same sign as  the cubes along $\mathcal R_1$, while taking the opposite sign of the cubes along $\mathcal R_2$. 

Let $C_v \in B^-$, with $C_v=\cup_{j=1}^N E(C_{(v, \ell)})$. Running a similar argument to the one given in $\R^3$ one obtains that all the cubes in $\mathcal E_{\text{ext}}(C_{(v, \ell)})\backslash C_{(v,\ell)}$  will glue to form a negative nodal domain $\Omega_v$ of $u_\ep$. We sketch the argument in what follows. All the negative cubes  in $\R^n \backslash C_v$ that touch $C_v$ do so through an edge in $\mathcal E_{\text{ext}} (C_v)$ since they will be at distance $1$ from the children structures $\{C_{(v, \ell)}\}_\ell$. Since the perturbation $f$ takes a strictly positive value (approx.  $+\ep$) along any edge in $\mathcal E_{\text{ext}} (C_v)$, the negative cubes  in $\R^n \backslash C_v$ will be separated from those in in $ C_v$.   Simultaneously, for each $\ell$, all the cubes in $E(C_{(v, \ell)}) \backslash C_{(v, \ell)}$ are glued to each other since they are negative cubes that touch $\mathcal E_{\text{ext}}(C_{(v, \ell)})$ and $\mathcal E_{\text{ext}}(C_{(v, \ell)})$ is a connected set on which the perturbation $f$  takes a strictly negative value (approx.  $-\ep$). This gives that $\mathcal E_{\text{ext}}(C_{(v, \ell)})$ belongs to a negative nodal domain of $u_\ep$, and that the negative cubes in $E(C_{(v, \ell)}) \backslash C_{(v, \ell)}$ are glued to the nodal domain after the perturbation is performed. Furthermore, two consecutive structures $E(C_{(v, \ell)}) $ and $E(C_{(v, \ell+1)}) $ are joined through an edge in  $\mathcal E_{\text{int}} (C_v)$.  This edge, which joins a negative cube in $E(C_{(v, \ell)}) $ and a negative cube in $E(C_{(v, \ell+1)}) $ has its boundary inside $\mathcal E_{\text{ext}}(C_{(v, \ell)})$. Since $f$ is strictly positive (approx. $+\ep$) on $\mathcal E_{\text{ext}}(C_{(v, \ell)})$, we know that the parts of the two cubes that are close to the boundary will be disconnected. However, since the perturbation was built so that $f$ is strictly negative  (approx. $-\ep$) at the midpoint of the edge, both negative cubes are glued to each other. In fact, one can build a curve $\gamma_\ell$ contained inside the nodal domain that joins $\mathcal E_{\text{ext}} (C_{(v, \ell)})$ with $\mathcal E_{\text{ext}} (C_{(v, \ell+1)})$. It then follows that all the cubes in  $\cup_{j=1}^N E(C_{(v, \ell)})\backslash C_{(v, \ell)}$ are glued to each other after the perturbation is performed and they will form the nodal domain $\Omega_v$ of $u_\ep$ containing  $\cup_{\ell=1}^n \mathcal E_{\text{ext}}(C_{(v, \ell)})$. 
One can carry the same stability arguments we presented in Section \ref{local} to obtain that at a local level there are no unexpected new nodal domains. For this to hold, as in the $\R^3$ case, the argument hinges on the fact that in the places where both $u_0$ and $f$ vanish, the gradient of $f$ is not zero (as explained at the end of each Rule).  Finally, Rule B is there to ensure that the topology of each nodal domain is controlled in the sense that when the cubes in $\mathcal E_{\text{ext}}(C_{(v, \ell)})\backslash C_{(v,\ell)}$  glue to each other  they do so without  creating unexpected handles. Indeed, the cubes in  $\mathcal E_{\text{ext}}(C_{(v, \ell)})\backslash C_{(v,\ell)}$ can be retracted to the set $\bigcup_{\ell=1}^N \Omega_{(v,\ell)}^* \cup \bigcup_{\ell=1}^{N-1}\gamma_\ell$ where $\Omega_{(v,\ell)}^*:= \overline{\Omega_{v, \ell}}\, \cup\, \bigcup_{\ell=1}^{N_\ell} \Omega_{(v,\ell, j)}^*$ and $\{(v, \ell, j):\, j=1,\dots, N_\ell\}$ are the children of $(v, \ell)$.

The argument we just sketched also shows that  the nodal domains $\Omega_v$ with $v \in T$ are nested as prescribed by the tree $T$. Indeed, claims (i), (ii) and (iii) in the proof of Theorem \ref{main theorem R b} are proved in $\R^n$ in exactly the same way we carried the argument in $\R^3$.


\end{document}